\documentclass[12pt,reqno,twoside]{amsart}
\usepackage{color}
\usepackage{amsmath, amsthm, amssymb}
\usepackage{amsfonts}

\usepackage{mathrsfs}
\usepackage{amscd}
\usepackage{wrapfig}
\usepackage{graphicx}
\usepackage[active]{srcltx}

\usepackage[ansinew]{inputenc}
\usepackage[dvips]{epsfig}
\usepackage{graphicx}
\usepackage[english]{babel}
\usepackage{hyperref}
\theoremstyle{plain}
\newtheorem{thm}{Theorem}[section]
\newtheorem{cor}[thm]{Corollary}
\newtheorem{lem}[thm]{Lemma}
\newtheorem{prop}[thm]{Proposition}

\theoremstyle{definition}

\theoremstyle{remark}
\newtheorem{rem}[thm]{Remark}

\numberwithin{equation}{section}

\newcommand{\average}{{\mathchoice {\kern1ex\vcenter{\hrule height.4pt
width 6pt depth0pt} \kern-9.7pt} {\kern1ex\vcenter{\hrule
height.4pt width 4.3pt depth0pt} \kern-7pt} {} {} }}

\def\R{\mathbb{R}}
\newcommand{\divv}{{\rm div}}
\newcommand{\logg}{{\rm log}\, }
\newcommand{\BB}{\mathcal {B}_r}
\newcommand{\BBone}{\mathcal {B}_1}
\newcommand{\BBB}{\mathcal {B}^{+}_r}
\newcommand{\B}{B_r}

\begin{document}

\title[The free boundary in the fractional obstacle problem]{Global regularity for the free boundary in the \\ obstacle problem for the fractional Laplacian}

\author[B. Barrios]{Bego\~na Barrios}
\address{ The University of Texas at Austin, Department of Mathematics, 2515 Speedway, Austin, TX 78751, USA.
Present address: University of La Laguna, Department of Analysis in Mathematics. C/. Astrof\'{\i}sico Francisco S\'{a}nchez s/n, 38200. Tenerife. Spain.}
\email{bbarrios@ull.es}

\author[A. Figalli]{Alessio Figalli}
\address{ETH Z\"urich, Department of Mathematics, R\"amistrasse 101, 8092 Z\"urich, Switzerland.}
\email{alessio.figalli@math.ethz.ch}

\author[X. Ros-Oton]{Xavier Ros-Oton}
\address{The University of Texas at Austin, Department of Mathematics, 2515 Speedway, Austin, TX 78751, USA}
\email{ros.oton@math.utexas.edu}

\thanks{BB  was partially supported by a postdoctoral fellowship
given by Fundaci\'on Ram\'on Areces (Spain) and MTM2013-40846-P, MINECO.
AF was supported by 
NSF Grants DMS-1262411 and DMS-1361122, and by
the ERC Grant ``Regularity and Stability in Partial Differential Equations (RSPDE)''. 
XR was supported by NSF grant DMS-1565186 and MINECO grant MTM2014-52402-C3-1-P (Spain)}


\begin{abstract}
We study the regularity of the free boundary in the obstacle problem for the fractional Laplacian under the assumption that the obstacle $\varphi$
satisfies $\Delta \varphi\leq 0$ near the contact region.
Our main result establishes that the free boundary consists of a set of regular points, which is known to be a $(n-1)$-dimensional $C^{1,\alpha}$ manifold by the results in \cite{CSS}, and a set of singular points, which we prove to be contained in a union of $k$-dimensional $C^1$-submanifold, $k=0,\ldots,n-1$.

Such a complete result on the structure of the free boundary was known only in the case of the classical Laplacian  \cite{C-obst1,C-obst2},
and it is new even for the Signorini problem (which corresponds to the particular case of the $\frac12$-fractional Laplacian).
A key ingredient behind our results is the validity of a new non-degeneracy condition $\sup_{B_r(x_0)}(u-\varphi)\geq c\,r^2$, valid at all free boundary points $x_0$.
\end{abstract}

\maketitle

\section{Introduction and main results}

\subsection{The obstacle problem for the fractional Laplacian}

Given a smooth function $\varphi :\R^n\to \R$,
the obstacle problem for the fractional Laplacian can be written as
\begin{equation}\label{pb}
\left\{
\begin{array}{l}
\min\bigl\{ u-\varphi,\,(-\Delta)^su\bigr\}=0\quad \textrm{in}\ \R^n,\\
\lim_{|x|\to\infty}u(x)=0,
\end{array}
\right.
\end{equation}
where
\[\qquad\qquad(-\Delta)^s u(x)=c_{n,s}\,\textrm{PV}\int_{\R^n}\bigl(u(x)-u(x+z)\bigr)\frac{dz}{|z|^{n+2s}},\qquad s\in(0,1),\]
is the fractional Laplacian. 

This kind of obstacle problems naturally appear when studying the optimal stopping problem for a stochastic process, and in particular they are used in the pricing of American options.
Indeed, the operator $(-\Delta)^s$ corresponds to the case where the underlying stochastic process is a stable radially symmetric L\'evy process.
We provide in the Appendix a brief informal description of the optimal stopping problem and its relation to \eqref{pb}.

In addition to this application, the obstacle problem for the fractional Laplacian appears in many other contexts,
for instance when studying the regularity of minimizers of some nonlocal interaction energies in  kinetic equations (see \cite{CDM}).

\subsection{Its local version and the Signorini problem}
Although the fractional Laplacian is a non-local operator, it is possible to localize the above problem via the so-called ``extension method''
(see \cite{MO,CSext})\footnote{The extension problem for the fractional Laplacian was first discovered by Molchanov and Ostrovskii \cite{MO}, and was known in the probability community since many years.
However, it seems that this paper went unnoticed in the PDE community.
We thank Rodrigo Ba\~nuelos for pointing out to us this reference.}.
More precisely, one adds an extra variable $y \in (0,\infty)$ and consider the function $\widetilde u(x,y)$ defined as the solution of
\[\left\{ \begin{array}{rcll}
\widetilde{u}(x,0)&=&u(x) &\textrm{in}\ \mathbb{R}^n,\\
L_a \widetilde{u}(x,y)&=& 0&\textrm{in}\ \mathbb{R}^{n+1}_{+},
\end{array}\right.\]
where $\R^{n+1}_+=\R^n\times (0,\infty)$ and
\begin{equation}\label{eq:La}
L_a\widetilde u:=-\divv_{x,y}\bigl(|y|^{a}\nabla_{x,y}\widetilde u\bigr),\qquad a:=1-2s.
\end{equation}
This function $\widetilde u$
can be obtained by minimizing the energy
$$
\min\biggl\{\int_{\R^{n+1}_+}|y|^a|\nabla_{x,y} v|^2\,dx\,dy\,:\, v(x,0)=u(x)\biggr\},
$$
and
 satisfies
\begin{equation}
\label{eq:frac y}
\lim_{y\downarrow 0} |y|^{a}\widetilde{u}_y(x,y)=(-\Delta)^su(x)\quad \textrm{in}\ \mathbb{R}^{n}.
\end{equation}
Moreover, $\widetilde{u}$ can be extended to the whole space $\R^{n+1}$ by even reflection, that is, $\widetilde{u}(x,y)=\widetilde{u}(x,-y)$, and then \eqref{pb}
becomes equivalent to
\begin{equation}\label{extpb}
\left\{ \begin{array}{rcll}
\widetilde u(x,0) &\geq&\varphi(x)&\quad\textrm{in}\ \R^n,\\
L_a \widetilde{u}&=&0& \quad \textrm{in}\ \mathbb{R}^{n+1}\setminus\left\{(x,0):\, u(x)=\varphi(x)\right\},\\
L_a \widetilde{u}&\geq&0& \quad \textrm{in}\ \mathbb{R}^{n+1},\\
\widetilde u(x,y)&\to &0&\quad \text{as $|(x,y)|\to\infty$}.
\end{array}\right.
\end{equation}
Notice that when $s=\frac12$ then $a=0$ and $L_a=\Delta_{x,y}$, so we recover the
Signorini problem (also called ``lower dimensional obstacle problem'' for the classical Laplacian).
This problem is interesting not only when stated on the whole space $\R^{n+1}$, but also in bounded
domains $\Omega\subset \R^{n+1}$ which are symmetric with respect to the hyperplane $\{y=0\}$
(see for instance \cite{Caf79,AC,ACS}).

Here we consider for simplicity the case when $\Omega=\BBone$
is the unit ball in $\R^{n+1}$, although our argument generalizes immediately to convex domains.
Since all our discussion holds for any $s\in (0,1)$,
we shall directly consider the local version of \eqref{pb} for all $s$, although
the most interesting case is when $s=\frac12$.
So, we set $B_1:=\BBone\cap \{y=0\}$, we consider
an obstacle $\varphi:B_1\to \R$ such that  $\varphi|_{\partial B_1}<0$, and we let $\widetilde u:\BBone\to \R$ be the solution of
\begin{equation}
\label{pb-sign}
\left\{ \begin{array}{rcll}
\widetilde u(x,0) &\geq&\varphi(x)&\textrm{on}\ \BBone\cap \{y=0\}, \\
L_a\widetilde u&=&0&\textrm{in}\ \BBone\setminus\bigl(\{y=0\}\cap\{u=\varphi\}\bigr), \\
L_a \widetilde u&\geq&0&\textrm{on}\ \BBone,\\
\widetilde u&=&0&\textrm{on}\ \partial \BBone,
\end{array}\right.
\end{equation}
where we use the notation $u(x)=\widetilde u(x,0)$.
Although often the assumption that $u$ vanished on ${\partial\BBone}$ is armless and it is just made to simplify the notation, for our results it will play a crucial role.

\begin{rem}
\label{rmk:u tilde u}
Since $\widetilde u$ is uniquely defined once its trace $u$ on $\{y=0\}$ is given,
by abuse of notation we shall say that a function $u:\R^n\to \R$ (resp. $u:B_1\to \R$) solves \eqref{extpb}
(resp. \eqref{pb-sign}) if its $L_a$-extension in $\R^{n+1}$ (resp. in $\BBone$) is a solution of  \eqref{extpb}
(resp. \eqref{pb-sign}).
\end{rem}

The main questions for both problems \eqref{extpb} and \eqref{pb-sign} are the regularity of $u$
and  the one of the boundary of the contact set $\{u=\varphi\}$ (also called ``free boundary'').
We next discuss the known regularity results on these questions.

\subsection{Known results}
Let us briefly discuss the known results about the regularity of the solution $u$ and of the free boundary $\partial \{u=\varphi\}$.
Before explaining the results in the fractional case, we first recall what is known in the case of the classical Laplacian.\smallskip

$\bullet$ {\bf The classical case.}
When $s=1$ (i.e., when the operator is the classical Laplacian) the obstacle problem \eqref{pb} is by now well understood.
Essentially, the main results establish that the solution $u \in C^{1,1}$,
and whenever $\Delta\varphi\leq -c_0<0$ then:\\
- the blow-up of $u$ at any free boundary point is a unique homogeneous polynomial of degree~2;\\
- the free boundary splits into the union of a set of regular points and a set of singular points;\\
- the set of regular points is an open subset of the free boundary of class $C^{1,\alpha}$;\\
- singular points are 
locally contained into a stratified union of $C^1$ submanifolds.\\
We refer to the classical papers of Caffarelli \cite{C-obst1,C-obst2} or to the recent book \cite{Sbook} for more details.
\smallskip

$\bullet$ {\bf The fractional case.}
For the fractional Laplacian, the first results obtained were for the Signorini problem
(corresponding to \eqref{extpb} or \eqref{pb-sign}
when $s=\frac12$): in \cite{AC} Athanasopoulos and Caffarelli obtain the optimal $C^{1,1/2}$ regularity of the solution $u$ when $\varphi\equiv 0$.
The general case $s \in (0,1)$ for  \eqref{extpb} and \eqref{pb-sign}  was investigated later by Silvestre and Caffarelli-Silvestre-Salsa \cite{S,CSS},
where the authors established, even for $\varphi\neq 0$, the optimal $C^{1+s}$ regularity.

Concerning the regularity of the free boundary,
in \cite{ACS} Athanasopoulos, Caffarelli, and Salsa investigate the Signorini problem for $\varphi\equiv 0$
and consider the so-called ``regular points'', consisting of the set of points
where the blow-up of $u$ at $x_0$ has homogeneity $1+s$ ($s=\frac12$), proving that this is
an open subset of the free boundary
of class $C^{1,\alpha}$.
This result was then extended to any $s \in (0,1)$ by Caffarelli, Salsa, and Silvestre \cite{CSS} for every smooth obstacle $\varphi\neq 0$.

Subsequently, again in the particular case $s=\frac12$, in \cite{GP} Garofalo and Petrosyan investigated the structure of
the set of singular points, that is the set of free boundary points where the contact set $\{u=\varphi\}$ has zero density,
proving a stratification results for such points in terms of the homogeneity of the blow-ups of $u$.
More recently, other generalizations and approaches to the Signorini and fractional obstacle problem have been investigated in \cite{GSVG,PP,GPSVG,FS}.

However, despite all these recent developments, the full structure of the free boundary was far from being understood: first of all, it was not even known whether the free boundary had Hausdorff dimension $n-1$.
In particular, a priori the free boundary could be a fractal set. Secondly,
the definitions of regular and singular points from \cite{CSS} and \cite{GP} do not exhaust all possible free boundary points, in the sense that they do not exclude the existence of other type of free boundary points,
and a priori the union of regular and singular points
may consists of a very small fraction of the whole free boundary.
Thus, the complete description of the free boundary was an open problem, even for
the case $s=\frac12$.

\subsection{Statement of the results}
The aim of this paper is to show that, for both problems \eqref{extpb} and \eqref{pb-sign}, under the assumption that the obstacle $\varphi$ satisfies $\Delta \varphi\leq  -c_0<0$,
regular and singular points {\em do exhaust} all possible free boundary points
(actually, in the case of the global problem \eqref{extpb} we only need the weaker inequality $\Delta \varphi \leq 0$).
In addition, we show that singular points are 
locally contained into a stratified union of $C^1$ submanifolds,
allowing us to obtain the same structure result on the free boundary as in the case $s=1$.

It is important to point out that, since our problem is non-local, one {does need} a {\em global} assumption in order to obtain such a result
(this in contrast with the case $s=1$, where all the assumptions are local).
In our case the global assumption is hidden in the fact that either we are considering the obstacle problem set in the whole $\R^n$,
or for the local case we are imposing zero  boundary conditions. It is easy to see that these assumptions are necessary
(see Remark \ref{rem-nondeg} below).
\smallskip

Before stating our result we introduce some notation: given an obstacle $\varphi$ and a solution $u$ to \eqref{extpb} or \eqref{pb-sign}
(recall the convention from Remark \ref{rmk:u tilde u}), we define the free boundary as
$$\Gamma(u):=\partial\{u=\varphi\}\subset \R^n.$$
A free boundary point $x_0\in\Gamma(u)$ is called {\it singular} if the contact set $\{u=\varphi\}$ has zero density at $x_0$, that is, if
\begin{equation}\label{singular}
\lim_{r\downarrow 0}\frac{\bigl|\{u=\varphi\}\cap \B(x_0)\bigr|}{|\B(x_0)|}=0.
\end{equation}
On the other hand, a free boundary point $x_0\in\Gamma(u)$ is called \emph{regular} if the homogeneity of the blow-up at $x_0$ is $1+s$.

We denote by $\Gamma_{1+s}(u)$ the set of regular points, and $\Gamma_2(u)$ the set of singular points
(the reason for  this notation will be clear from the theorem below).
Our main result is the following:

\begin{thm}\label{thm1}
Let $u$ be:
\begin{itemize}
\item[(A)] either a solution to the ``global'' problem \eqref{extpb}, with $\varphi:\R^n\to \R$ satisfying
\begin{equation}\label{obstacle g}
\varphi\in C^{3,\gamma}(\R^n),\qquad \Delta\varphi\leq 0\quad\textrm{in}\ \{\varphi>0\},\qquad \emptyset \neq \{\varphi>0\}\subset\subset \R^n,
\end{equation}
for some $\gamma>0$;
\item[(B)] or a solution to the ``local'' problem \eqref{pb-sign}, with $\varphi:B_1\to \R$ satisfying
\begin{equation}\label{obstacle l}
\varphi\in C^{3,\gamma}(B_1),\qquad \Delta\varphi\leq -c_0<0\quad\textrm{in}\ \{\varphi>0\},\qquad  \emptyset \neq \{\varphi>0\}\subset\subset B_1,
\end{equation}
for some $c_0>0$ and $\gamma>0$.
\end{itemize}
Then, at every singular point the blow-up of $u$ is a homogeneous polynomial of degree $2$,
and the free boundary can be decomposed as
\[\Gamma(u)=\Gamma_{1+s}(u)\cup \Gamma_2(u),\]
where $\Gamma_{1+s}(u)$ (resp. $\Gamma_2(u)$) is a open (resp. closed) subset of $\Gamma(u)$.

Moreover, $\Gamma_{1+s}(u)$ is a $(n-1)$-dimensional manifold of class $C^{1,\alpha}$,
while $\Gamma_2(u)$ can be stratified as the union of $\{\Gamma_2^k(u)\}_{k=0,\ldots,n-1}$,
where  $\Gamma_2^k(u)$ is contained in a $k$-dimensional manifold of class $C^1$.
\end{thm}

As mentioned above, the $C^{1,\alpha}$ regularity of $\Gamma_{1+s}(u)$ was established in \cite{CSS},
so the main point of our result is the fact that $\Gamma(u)\setminus \Gamma_{1+s}(u)$ consists only of singular points,
that the blow-up of $u$ at these points has homogeneity $2$, and that $\Gamma_2(u)$ can be stratified into $C^1$ submanifolds. We remark that, if the obstacle is $C^\infty$, then also the set $\Gamma_{1+s}(u)$ is of class $C^\infty$ \cite{dSS,JN,KRS}.

A key ingredient in the proof of Theorem \ref{thm1} is the non-degeneracy condition
\begin{equation}
\label{eq:non deg}
\sup_{\B(x_0)}(u-\varphi)\geq c\,r^2,
\end{equation}
which we prove at all free boundary points $x_0$.
Thanks to this fact we can show that the homogeneity $m$ of any blow-up satisfies $m\leq 2$.
Then, by the results of \cite{CSS} we know that $m<2$ implies in fact that $m=1+s$, and thus we only have to study the set of free boundary points with homogeneity $m=2$.
For this, building upon some ideas used in \cite{GP} for the case $s=\frac12$, we show that the set $\Gamma_2(u)$ is contained in a union of $k$-dimensional $C^1$ submanifolds, $k=0,\ldots, n-1$.
Notice that in \cite{GP} the authors could only show that the set of singular free boundary points with even homogeneity is contained into a {\em countable} union of submanifolds,
while our result proves that, for any $k=0,\ldots, n-1$, there exists {\em one} $k$-dimensional $C^1$ manifold (not necessarily connected) which covers the whole set $\Gamma_2^k(u)$.
\smallskip

The paper is organized as follows.
In Section~\ref{sec-prel} we introduce some definition and show some basic properties of solutions.
In Section~\ref{sec2} we prove the non-degeneracy at free boundary points.
In Section~\ref{sec3} we prove an Almgren-type frequency formula.
In Section~\ref{sec4} we show that blow-ups are homogeneous of degree either $1+s$ or $2$.
In Section~\ref{sec5} we prove a Monneau-type monotonicity formula that extend the one obtained in \cite{GP} for $s=\frac12$.
Finally, in Section~\ref{sec6} we show uniqueness of blow-ups and establish Theorem~\ref{thm1}.
\smallskip

{\bf Acknowledgments.} BB has been partially supported by a postdoctoral fellowship
given by Fundaci\'on Ram\'on Areces (Spain) and MTM2013-40846-P, MINECO. AF has been partially supported
by NSF Grant DMS-1262411 and NSF Grant DMS-1361122.

\section{Preliminaries}
\label{sec-prel}

\subsection{Existence and uniqueness of solutions}
Although this is not the focus of our paper, we make few comments about the existence of solutions to problems
\eqref{extpb} and \eqref{pb-sign} under the assumption that $\varphi:\R^n\to \R$ is a continuous function
satisfying either $\{\varphi>0\}\subset\subset \R^n$ or $\{\varphi>0\}\subset\subset B_1$,
depending on the problem we are considering.

The existence and uniqueness of solutions to \eqref{pb-sign} is standard: $\widetilde u$
is the unique minimizer of the variational problem
$$
\min\biggl\{\int_{\BBone}|y|^a|\nabla_{x,y} v|^2\,dx\,dy\,:\, v(\cdot,0)\geq \varphi,\,v|_{\partial\BBone}=0\biggr\}.
$$
On the other hand, one must be a bit more careful with \eqref{extpb}:
a possible way to construct a solution is to consider the limit as $R\to \infty$ of the minimizers $\widetilde u_R$ to the problem
$$
\min\biggl\{\int_{\mathcal B_R}|y|^a|\nabla_{x,y} v|^2\,dx\,dy\,:\, v(\cdot,0)\geq \varphi,\,v|_{\partial \mathcal B_R}=0\biggr\}.
$$
It is easy to check (by a comparison principle)
that $\widetilde u_R \geq 0$ in $\mathcal B_R$ and that $\widetilde u_R \leq \widetilde u_{R'}|_{\mathcal B_R}$ if $R\leq R'$,
so $\widetilde u:=\lim_{R\to \infty}\widetilde u_R$ exists and solves the first three relations in \eqref{extpb}.
The only nontrivial point is whether $\widetilde u$ vanishes at infinity, and this is actually false in some cases
(this fact was already observed in \cite{S}).

To understand this point, consider first the simpler case $n=1$ and $s=1$ (even if we only consider $s \in (0,1)$,
the case $s=n=1$ allows one to understand what is happening).
This corresponds to take the limit of $u_R$ as $R\to \infty$, where $u_R$ minimizes the Dirichlet energy in $B_R$
among all functions $v\geq \varphi$ vanishing on $\partial B_R$.
Since the functions $u_R$ are harmonic (hence linear) outside the contact region $\{u_R=\varphi\}$, it is not difficult to check that the limit of $u_R$
is  the constant function $u\equiv \max_\R \varphi$. In particular we see that $u \not\to 0$ as $|x|\to \infty$.

We now show that
\begin{equation}
\label{eq:vanish}
\widetilde u(x,y)\to 0\quad \text{as $|(x,y)|\to\infty$}\qquad \text{for }
\left\{
\begin{array}{ll}
s \in \left(0,\frac12\right) &\text{if $n=1$,}\\
s \in (0,1) &\text{if $n\geq 2$}.
\end{array}
\right.
\end{equation}
To see this, consider the fundamental solution of $(-\Delta)^s$ given (up to a multiplicative constant) by
$$
G_{n,s}(x):=\frac{1}{|x|^{n-2s}},\qquad x\in\mathbb{R}^{n}\setminus\{0\},
$$
and extends it to $\R^{n+1}$ as
$$
\widetilde G_{n,s}(x,y):=\int_{\R^n}G_{n,s}(z) \,\frac{|y|^{2s}}{\left(|x-z|^2+y^2\right)^{(n+2s)/2}}\,dz.
$$
It can be easily checked that this functions satisfies
$$
L_a \widetilde G_{n,s}=0 \quad \text{in $\R^{n+1}\setminus \{0\}$},
\qquad L_a \widetilde G_{n,s} \geq 0 \quad \text{in $\R^{n+1}$,}
$$
(see for instance \cite[Section 2.4]{CSext}), so we can use this function as a barrier for our solution $\widetilde u$.
More precisely, consider $M>1$ large enough so that $M\,G_{n,s} \geq \varphi$.
Then it follows by comparison that $\widetilde u_R \leq M\,\widetilde G_{n,s}$ in $\R^{n+1}$
for all $R>0$, and letting $R\to \infty$ we obtain that $\widetilde u\leq M\, \widetilde G_{n,s}$. Since the latter function vanishes at infinity when $s \in (0,\frac12)$ and $n=1$ or $s \in (0,1)$ and $n\geq 2$,
this proves \eqref{eq:vanish}.

The discussion above shows why the cases $s\in[\frac12,1)$ and $n=1$ are critical: the Green function does not vanish at infinity.
Still it is worth noticing that, even in these cases, one could generalize our results to the
solution constructed above as the monotone limit of the functions $\widetilde u_R$.
However, to have a cleaner statement with precise boundary conditions at infinity, we have preferred to state our results for solutions to \eqref{extpb}.

\subsection{A useful transformation}
In Section~\ref{sec3} we shall prove an Almgren-type frequency formula to study free boundary points where the blow-ups of our solution $u$
have homogeneity at most $2$. While the study of free boundary points where the homogeneity is strictly less than $2$ was essentially done in \cite{CSS},
to investigate points with homogeneity equal to $2$ it will be important to replace $\widetilde u-\varphi$
with a suitable variant of it for which the $L_a$ operator is very small.

More precisely, given $\widetilde u$ solving either \eqref{extpb} or \eqref{pb-sign},
given a free boundary point $x_0\in\Gamma(u)$ we define
\begin{equation}\label{v}
v^{x_0}(x,y):=\widetilde{u}(x,y)-\varphi(x)+\frac{1}{2(1+a)}\left\{\Delta\varphi(x_0)\,y^2+(\nabla\Delta\varphi)(x_0)\cdot (x-x_0)\,y^2\right\}.
\end{equation}
First of all notice that
\[v^{x_0}(x,0)=u(x)-\varphi(x)\]
for all $x\in \R^n$, hence $v^{x_0}(x,0)\geq0$.
Furthermore, by \eqref{extpb} and \eqref{obstacle g} (resp. \eqref{pb-sign} and \eqref{obstacle l}), we have
\begin{eqnarray}\label{ext2}
|L_a v^{x_0}(x,y)|&=&|y|^{a}\bigl|-(\Delta_x\varphi)(x)+(\Delta_x\varphi)(x_0)+ (x-x_0)\cdot\nabla(\Delta_x\varphi)(x_0))\bigr|\nonumber\\
&\leq& C\,|y|^{a}|x-x_0|^{1+\gamma},
\end{eqnarray}
for every $(x,y)\in \mathbb{R}^{n+1}\setminus\left\{(x,0):\, v^{x_0}(x,0)=0\right\}$
(resp. for every $(x,y)\in \BBone\setminus\left\{(x,0):\, v^{x_0}(x,0)=0\right\}$).
Finally it is also important to observe that, since $\varphi\in C^{3,\gamma}$, then $v^{x_0}$ depends continuously on $x_0$.\\

Throughout the paper, we shall use $v$ instead of $v^{x_0}$ whenever the dependence on point $x_0$ is clear.
Also, given a point $x_0 \in \R^n$, we denote by $\BB(x_0,0)$ the ball in $\mathbb{R}^{n+1}$ of radius $r$ centered at $(x_0,0)$, by $\BBB(x_0,0)$ the upper
half ball $\BB(x_0,0)\cap\mathbb{R}^{n+1}_{+}$, and by $\B(x_0)$ the ball in $\R^n$ of radius $r$ centered at $x_0$.

\section{Nondegeneracy}
\label{sec2}

As mentioned in the introduction,
a key ingredient in the proof of Theorem \ref{thm1} is the validity of the non-degeneracy condition
\eqref{eq:non deg} at every free boundary point $x_0$.
We now give two different proofs of it, depending whether we are in the global or in the local situation.

We begin with the global case.
Recall that problems \eqref{extpb} and \eqref{pb} are equivalent, so we can use either formulations.

\begin{lem}\label{nondegeneracy0}
Let $u$ solve the obstacle problem \eqref{pb}, with $\varphi$ satisfying~\eqref{obstacle g}.
Then there exist constants $c_1,r_1>0$ such that the following holds:
for any $x_0\in \Gamma(u)$ we have
\[\sup_{\B(x_0)}\left(u-\varphi\right) \geq c_1r^2 \qquad \forall\,r \in (0,r_1).\]
\end{lem}
\begin{proof}
Let us consider the function $w:=(-\Delta)^{s} u$.
Notice that $w \geq 0$, and $w$ cannot be identically $0$ as otherwise $u$ would be a $s$-harmonic function globally bounded on $\R^n$,
hence constant. Since $u$ vanishes at infinity this would imply that $u\equiv 0$, which is impossible since by assumption $\emptyset \neq \{\varphi>0\}\subset \{u>0\}$.

Since $w=(-\Delta)^su$ vanishes in the set $\{u>\varphi\}$,
for any point $x_1 \in \{u>\varphi\}$ we have
$$
(-\Delta)^{1-s} w(x_1)=c_{n,1-s}\,\textrm{PV}\int_{\R^n}\frac{-w(z)\, dz}{|x_1-z|^{n+2(1-s)}}<0,$$
where the last inequality follows from the fact that $w \gneqq0$.
In particular, by compactness, we see that there exist constants $c_0,r_0>0$ such that
$$
(-\Delta)^{1-s} w(x_1) \leq -c_0<0
$$
for any $x_1 \in \{u>\varphi\}$ with ${\rm dist}\bigl(x_1,\Gamma(u)\bigr) \leq r_0$.
Now, since $u$ is a global solution, by the semigroup property of the fractional Laplacian we obtain that
\begin{equation}\label{trivial}
(-\Delta) u= (-\Delta)^{1-s} w \leq -c_0\qquad  \text{in } U_0:=\{u>\varphi\}\cap \left\{{\rm dist}\bigl(\cdot ,\Gamma(u)\bigr) \leq r_0\right\} .
\end{equation}
We now observe that, since $u>0$ on the contact set,
again by compactness there exists a constant $h_0>0$ such that
$$
\varphi\geq h_0\qquad \text{in }\{u=\varphi\}.
$$
In particular, by continuity of $\varphi$, there exists $r_1 \in (0,\frac{r_0}2)$ such that
\begin{equation}\label{trivial2}
\varphi>0\qquad  \text{in } U_1:=\{u>\varphi\}\cap \left\{{\rm dist}\bigl(\cdot ,\Gamma(u)\bigr) \leq 2\,r_1\right\}.
\end{equation}
Consider now an arbitrary point $x_1 \in U_1$ with ${\rm dist}\bigl(\cdot ,\Gamma(u)\bigr) \leq r_1$, and take $r \in (0,r_1)$.
Since $U_1\subset U_0$, it follows by \eqref{trivial}, \eqref{trivial2}, and \eqref{obstacle g}, that the function $u_1:=u-\varphi$ satisfies
$$\Delta u_1\geq c_0\qquad \text{in }\{u_1>0\}\cap B_r(x_1)=:\Lambda\cap B_r(x_1),$$
or equivalently
$$u_2(x):=u_1(x)-\frac{c_0}{2n}|x-x_1|^{2}$$
is sub-harmonic in $\Lambda\cap B_r(x_1)$. Hence, by the maximum principle,
$$
0<u_1(x_1)\leq \sup_{\Lambda\cap B_r(x_1)}u_2=\sup_{\partial \left(\Lambda\cap B_r(x_1)\right)}u_2,
$$
and noticing that $u_2<0$ on $\partial{\Lambda}\cap B_r(x_1)$ we deduce that
$$0<\sup_{\Lambda \cap \partial B_r(x_1)}u_2\leq \sup_{\partial B_r(x_1)}u_1-c_1r^2,\qquad c_1:=\frac{c_0}{2n}.$$
Since $x_1 \in U_1$ was arbitrary, the result follows by letting $x_1\rightarrow x_0$ .
\end{proof}

We now consider the local case.
As we shall see we now need a slightly stronger assumption on the obstacle, namely that $\Delta \varphi\leq -c_0<0$.
Notice that this is exactly the same assumption needed in the obstacle problem for the classical Laplacian,
and it is a peculiarity of the global problem and the non-locality of the fractional Laplacian
that allowed us to weaken this hypothesis in the previous lemma.

\begin{lem}\label{nondegeneracy}
Let $u$ solve the obstacle problem \eqref{pb-sign}, with $\varphi$ satisfying~\eqref{obstacle l}.
Then there exist constants $c_1,r_1>0$ such that the following holds:
for any $x_0\in \Gamma(u)$ we have
\[\sup_{\B(x_0)}\left(u-\varphi\right) \geq c_1r^2 \qquad \forall\,r \in (0,r_1).\]
\end{lem}

\begin{proof}
As in the proof of Lemma \eqref{nondegeneracy0} we observe that,
since $u>0$ on the contact set,
by compactness there exists a constant $h_0>0$ such that
$$
\varphi\geq h_0\qquad \text{in }\{u=\varphi\}.
$$
In particular, by  \eqref{obstacle l} and the continuity of $\varphi$, there exists $r_1>0$ such that
$$
\varphi>0\qquad  \text{in } U_1:=\{u>\varphi\}\cap \left\{{\rm dist}\bigl(\cdot ,\Gamma(u)\bigr) \leq 2\,r_1\right\}\subset \subset B_1.
$$
Consider now an arbitrary point $x_1 \in U_1$ with ${\rm dist}\bigl(\cdot ,\Gamma(u)\bigr) \leq r_1$,
define the function
\[w(x,y):=\widetilde u(x,y)-\varphi(x)-\frac{c_0}{2n+2(1+a)}\bigl(|x-x_1|^2+y^2\bigr)\]
where $c_0$ is the constant in \eqref{obstacle l}, and consider $r \in (0,r_1)$.

Since
$$\mbox{$L_a \widetilde u=0$\quad in $\{y\neq0\}\cup\bigl\{ \{y=0\}\cap\{u>\varphi\} \bigr\}$}$$
we get that
\[L_aw(x,y)=L_a \widetilde u(x,y)-|y|^{a}\bigl(\Delta\varphi(x)+c_0\bigr)\geq 0\quad \textrm{in}\ U,\]
where
\[U:=\BB(x_1,0)\setminus\bigl( \{y=0\}\cap\{u=\varphi\}\bigr)\subset \subset \BBone.\]
Hence,
since $w<0$ in $ \{y=0\}\cap\{u=\varphi\}$ and $w(x_1,0)>0$, the maximum principle
yields
\begin{equation}\label{xx}
0<w(x_1)\leq \sup_{U}w=\sup_{\partial U}w.
\end{equation}
Noticing that $\partial U=\partial{\BB}(x_1,0)\cup \bigl( \{y=0\}\cap\{u=\varphi\}\bigr)$ and $w<0$ in $\bigl( \{y=0\}\cap\{u=\varphi\}\bigr)$, it follows by \eqref{xx} that
\[\sup_{\partial \BB(x_1,0)}w>0,\]
so, letting $x_1\rightarrow x_0$, we get
\begin{equation}
\label{sup BB}
\sup_{\BB(x_0,0)}(\widetilde u-\varphi)-\frac{c_0}{2n+2(1+a)}\,r^2\geq \sup_{\partial \BB(x_0,0)}(\widetilde u-\varphi)-\frac{c_0}{2n+2(1+a)}\,r^2\geq0.
\end{equation}
To conclude the proof we now show that this supremum is attained for $\{y=0\}$.

To prove this we first notice that, by symmetry, $\widetilde u(x,y)=\widetilde u(x,-y)$.
Also, since $L_a \widetilde u \geq 0$, $L_a\widetilde u=0$ outside the contact set $\{u=\varphi\}$, and $\widetilde u|_{\partial \BBone}=0$, it
follows by the maximum principle that $\widetilde u \geq 0$ in $\BBone$
and that $\widetilde u$ attains its maximum on the contact set $\{u=\varphi\}$.
Hence, using again that $\widetilde u=0$ on $\partial\BBone$, we deduce that
$$
|y|^a\partial_y\widetilde  u \leq 0 \quad \text{on }\partial\BBone\cap\{y>0\},
\qquad \lim_{y\to 0^+} |y|^a\partial_y\widetilde  u(x,y) \leq 0 \quad \text{on }\{u=\varphi\}.
$$
Also, since $u$ is even in $y$ and it is smooth outside the contact set,
we have
$$
\lim_{y\to 0^+} |y|^a\partial_y\widetilde  u(x,y)= 0 \quad \text{on }\{u>\varphi\}.
$$
Thus, since the function $y^a\partial_y\widetilde u$
is $L_{-a}$-harmonic in $\BBone^+$ (as can be easily checked by a direct computation), it follows by the maximum principle that
$y^a\partial_y \widetilde u\leq0$ in $\mathcal {B}_1^+$, that is,
$\widetilde u$ is decreasing with respect to $y$ inside $\BBone^+$.
Since $\widetilde u$ is even in $y$
this proves that
$$
\widetilde u(x,y)\leq \widetilde u(x,0)=u(x)\qquad \forall\,(x,y)\in \BBone,
$$
and \eqref{sup BB} yields
$$\sup_{\B(x_0)}(u-\varphi)=\sup_{\BB(x_0,0)}(\widetilde{u}-\varphi)\geq \frac{c_0}{2n+2(1+a)}\,r^2,$$
as desired.
\end{proof}

\begin{rem}\label{rem-nondeg}
It is worth noticing that the proof of Lemma \ref{nondegeneracy} works also in $\R^{n+1}$,
but has the drawback (with respect to Lemma \ref{nondegeneracy0}) of requiring that $\Delta \varphi\leq -c_0<0$.

In any case, it is important to observe that the non-degeneracy estimate at free boundary points
is a ``global'' property, in the sense that it
crucially relies on the fact that we are either studying the obstacle problem in the whole $\R^n$
or we are assuming zero boundary conditions:
indeed, while \eqref{sup BB} always holds independently of the value of $\widetilde u$ on $\partial \BBone$,
we need to know that $\R^+\ni y\mapsto \widetilde u(x,y)$ is decreasing
to prove non-degeneracy   (see the proof of Lemma \ref{nondegeneracy}).

To show that non-degeneracy does \emph{not} holds in bounded domains even if one assumes the obstacle $\varphi$ to be $C^\infty$ and uniformly concave,
consider a $s$-harmonic function $u$ in $B_1$ which satisfies $D^2u\leq -{\rm Id}$ in $B_{1/2}$ (the existence of such a function follows for instance from \cite{DSV}).
Then, by taking as obstacle a smooth function $\varphi\leq u$ satisfying $D^2\varphi\leq -{\rm Id}$, $u(0)=\varphi(0)$, and such that $\sup_{B_r}(u-\varphi)\leq r^{30}$ for $r$ small,
we see that non-degeneracy fails.
\end{rem}

We now want to transfer these non-degeneracy informations to the function $v(x,y)=v^{x_0}(x,y)$ defined in \eqref{v}.
Before doing that, we need the following weak-Harnack estimate:

\begin{lem}\label{fabes-kenig}
Let $z \in \R^{n+1}$, $r>0$, and let
$w$ satisfy $L_aw\leq |y|^a f$ in $\BB(z)$.
Then
\[\sup_{\mathcal{B}_{r/2}(z)}w\leq C\left(\frac{1}{r^{n+1+a}}\int_{\BB(z)}|y|^{a}w^2\,dx\,dy\right)^{1/2}+C\,r^2\|f\|_{L^\infty(\BB(z))},\]
for some constant $C$ depending only on $n$ and $s$.
\end{lem}

\begin{proof}
The result follows from the classical elliptic estimates of Fabes, Kenig, and Serapioni~\cite{FKS}.
Namely, write $z=(\bar x,\bar y) \in \R^n\times \R$ and  define
\[\psi(x,y):=\frac{1}{2(n+1+a)}\|f\|_{L^\infty(\BB(z))}\bigl(|x-\bar x|^2+(y-\bar y)^2\bigr),\]
so that $L_a\psi= -|y|^a\|f\|_{L^\infty(\BB(z))}$.

In this way $\omega:=w+\psi$ satisfies $L_a\omega \leq 0$,
and by  \cite[Theorem 2.3.1]{FKS} we get
\[\sup_{\mathcal{B}_{r/2}(z)}\omega \leq C\left(\frac{1}{r^{n+1+a}}\int_{\BB(z)}|y|^{a}\omega^2\,dx\,dy\right)^{1/2}.\]
The result  then follows easily noticing that $0 \leq \psi \leq Cr^2\|f\|_{L^\infty(\BB(z))}$ inside $\mathcal{B}_{r}(z)$.
\end{proof}

We can now prove the following:

\begin{cor}\label{non-degeneracy-v}
Let $u$ solve:
\begin{itemize}
\item[(A)] either the obstacle problem \eqref{extpb}, with $\varphi$ satisfying~\eqref{obstacle g};
\item[(B)] or the obstacle problem \eqref{pb-sign}, with $\varphi$ satisfying~\eqref{obstacle l}.
\end{itemize}
Let $c_1,r_1$ be as in Lemma \ref{nondegeneracy0} (in case (A)) or Lemma \ref{nondegeneracy} (in case (B)).
Also, let $x_0\in \R^n$ be a free boundary point, and let $v=v^{x_0}$ be defined as in \eqref{v}.
Then
\begin{equation}\label{corolar}
\sup_{\B(x_0)}v(x,0) \geq c_1r^2\qquad \forall\,r \in (0,r_1).
\end{equation}
Moreover there exist positive constants $c_2$ and $r_2$, independent of $x_0$,
such that
\begin{equation}\label{tormenta}
\int_{\BB(x_0,0)}|y|^{a}|v(x,y)|^2\,dx\,dy \geq c_2\,r^{n+a+5}\qquad \forall\,r \in (0,r_2).
\end{equation}
\end{cor}

\begin{proof}
Since $v(x,0)=u(x)-\varphi(x)$, \eqref{corolar} follows immediately from Lemmas \ref{nondegeneracy0}-\ref{nondegeneracy}, so we only need to prove \eqref{tormenta}.

For this we define the function $v^+:=\max\{v,0\}$ in $\R^{n+1}$ and notice that,
by \eqref{ext2},
\[L_av^+(x,y)\leq C|y|^a|x-x_0|^{1+\gamma}\quad\textrm{in}\ \{v^+>0\}.\]
Since $v^+ \geq 0$ we see that $L_av^+\leq 0$ in the set $\{v^+=0\}$, therefore
\begin{equation}\label{p18}
L_av^+\leq C|y|^a|x-x_0|^{1+\gamma}\quad\textrm{in}\ \R^{n+1}.
\end{equation}
This allows us to apply Lemma \ref{fabes-kenig} to deduce that
\[\sup_{\mathcal{B}_{r/2}(z)}v^+\leq C\left(\frac{1}{r^{n+1+a}}\int_{\BB(z)}|y|^{a}\bigl|v^+\bigr|^2\,dx\,dy\right)^{1/2}+C\,r^{3+\gamma},\]
for any $z\in \R^{n+1}$.
In particular, applying the estimate above with $z=(x_0,0)$,  \eqref{corolar} gives
\[C\left(\frac{1}{r^{n+1+a}}\int_{\BB(x_0,0)}|y|^{a}\bigl|v^+\bigr|^2\,dx\,dy\right)^{1/2}\geq c_1\,\Bigl(\frac{r}2\Bigr)^2-C\,r^{3+\gamma},\]
hence
\[\int_{\BB(x_0,0)}|y|^{a}|v(x,y)|^2\,dx\,dy \geq c_2 \,r^{n+a+5}\]
for $r$ small enough, as desired.
\end{proof}

\section{Frequency formula}
\label{sec3}

The main objective of this section is to establish an Almgren-type frequency formula similar to the ones in \cite{CSS,GP}. 
More precisely we prove the following:

\begin{prop}\label{Almgren}
Let $u$ solve:
\begin{itemize}
\item[(A)] either the obstacle problem \eqref{extpb}, with $\varphi$ satisfying~\eqref{obstacle g};
\item[(B)] or the obstacle problem \eqref{pb-sign}, with $\varphi$ satisfying~\eqref{obstacle l}.
\end{itemize}
Let $x_0\in \Gamma(u)$ be a free boundary point, let $v=v^{x_0}$ be defined as in \eqref{v},
and set
\begin{equation}\label{Hx_0}
\mathcal{H}^{x_0}(r,v):=\int_{\partial \BB(x_0,0)}{|y|^{a}v^2 }.
\end{equation}
Also, let $\gamma>0$  be as in \eqref{obstacle g}-\eqref{obstacle l}.
Then there exist constants $C_0,\, r_0>0$, independent of $x_0$, such that the function
$$r\mapsto \Phi^{x_0}(r,v):=\bigl(r+C_0\,r^2\bigr)\,\frac{d}{dr}\log \max\left\{\mathcal{H}^{x_0}(r,v),\ r^{n+a+4+2\gamma}\right\},$$
is monotone nondecreasing on $(0,r_0)$.
In particular the limit $\lim_{r\downarrow0}\Phi^{x_0}(r,v):=\Phi^{x_0}(0^+,v)$ exists.
\end{prop}

To simplify the notation we shall denote $\Phi=\Phi^{x_0}$ and $\mathcal{H}=\mathcal{H}^{x_0}$ when no confusion is possible.
We notice that the result above is a modification of the one established in \cite[Theorem 3.1]{CSS},
which corresponds to the case $\gamma=0$ in our Proposition \ref{Almgren}.
There the authors can assume $\gamma=0$ since they only study free boundary points where the blow-ups
of $u$ have homogeneity strictly less than $2$. However, since our main focus is to
study free boundary points with homogeneity 2, we need to add a parameter $\gamma>0$ (and, for the same reason,
we need to consider the function $v$ instead of $\widetilde u-\varphi$).
Notice that, in the case $s=\frac12$, a similar Almgren-type frequency formula was used in \cite{GP} to study singular points of homogeneity $2m$ with $m$ integer.
When $m=1$, their frequency formula corresponds to $\gamma=1$ in Proposition \ref{Almgren}.
\\

To prove Proposition \ref{Almgren} we can assume without loss of generality that $x_0=0$.
We will denote by $\partial \BB$ (resp. $\BB$) the sets $\partial \BB(x_0,0)=\partial \BB(0,0)$ (resp. $\BB(x_0,0)=\BB(0,0)$).
Also, we shall use primes to denote derivatives with respect to $r$.

Before proving Proposition \ref{Almgren} we establish an auxiliary lemma that provides us with some upper bounds for the functions
\begin{equation}\label{HG}
\mathcal{G}(r,v):=\int_{\BB}{|y|^{a}v^2 }\qquad\mbox{and}\qquad \mathcal{H}(r,v)=\int_{\partial \BB}{|y|^{a}v^2 }=\mathcal{G}'(r,v).
\end{equation}

\begin{lem}\label{auxiliar_frecuency}
Let $v$ be as in Proposition \ref{Almgren}, and define
\begin{equation}\label{D}
\mathcal{D}(r,v):=\int_{\BB}{|y|^{a}|\nabla v|^{2}}.
\end{equation}
Then there exist constants $\bar{C},\, \bar{r}>0$, independent of $x_0$, such that
\begin{equation}\label{H1}
\mathcal{H}(r,v)\leq \bar{C}\left(r\,\mathcal{D}(r,v)+ r^{n+a+6+2\gamma}\right)\qquad \textrm{for all}\ r\in (0,\bar{r}),
\end{equation}
and
\begin{equation}\label{G1}
\mathcal{G}(r,v)\leq \bar{C}\left(r^2\,\mathcal{D}(r,v)+ r^{n+a+7+2\gamma}\right)\qquad \textrm{for all}\ r\in (0,\bar{r}).
\end{equation}
\end{lem}

\begin{proof}
Notice that, by our assumption on the positivity set of $\varphi$ (see \eqref{obstacle g}-\eqref{obstacle l}),
the contact set $\{u=\varphi\}$ is compact. In particular, in the local case, there exists $\bar{r}>0$ such that $B_{\bar{r}}(x_0)\subset B_1$ for every $x_0\in\Gamma(u)$.

Let us consider $x_0=0$ and $0<r<\bar{r}$. Then by \cite[Lemma 2.9]{CSS} it follows that
$$v(0)\geq \frac{1}{\omega_{n+a}r^{n+a}}\int_{\partial \BB}{|y|^{a} v}- C\, r^{3+\gamma},$$
so one can follow the proof of \cite[Lemma 2.13]{CSS} to get
$$
\int_{\partial \BB}{|y|^{a} v^2}\leq C\, r\int_{\BB}{|y|^{a}|\nabla v|^{2}}+ C \,r^{(n+a)+6+2\gamma}.
$$
The previous inequality proves \eqref{H1}, and integrating it with respect to $r$ we obtain \eqref{G1}.
\end{proof}

Following the ideas developed in \cite{CSS}, we now prove the main result of this section.

\begin{proof}[Proof of Proposition \ref{Almgren}]
As observed in \cite[Proof of Theorem 3.1]{CSS},
in order to prove that $\Phi(r,v)$ is increasing one can
concentrate in each of the two values for the maximum separately.

Since in the case
$$
\Phi(r,v)=\bigl(r+C_0r^2\bigr)\,\frac{d}{dr}\log r^{n+a+4+2\gamma}=(1+C_0r)(n+a+4+2\gamma)
$$
the function $\Phi(\cdot,v)$ is clearly monotonically increasing, it is enough
to prove that $\Phi'(r,v)\geq0$ in the case $\mathcal{H}(r,v)>r^{n+a+4+2\gamma}$.

First of all we notice that, since
$$\mathcal{H}(r,v)=r^{n+a}\int_{\partial\BBone}{|y|^{a}v^{2}(rx,ry)},$$
it follows that
\begin{eqnarray}
\mathcal{H}'(r,v)&=&(n+a)\,\frac{\mathcal{H}(r,v)}{r}+2\, r^{n+a}\int_{\partial\BBone}{|y|^{a}v(rx,ry)\,\nabla v(rx,ry)\cdot(x,y)}\nonumber\\
&=&(n+a)\,\frac{\mathcal{H}(r,v)}{r}+2\,\mathcal{I}(r,v),\label{h'}
\end{eqnarray}
where
\begin{eqnarray}
\mathcal{I}(r,v)&:=&\int_{\partial \BB}{|y|^{a} v \,v_{\nu}}=\mathcal{D}(r,v)+\int_{\BB}{v\, \divv (|y|^{a}\nabla v)}\nonumber \\
&=&\mathcal{D}(r,v)-\int_{\BB}{v\, L_av}
\label{i}
\end{eqnarray}
(recall that $\divv (|y|^{a}\nabla v)=-L_a v$).
Hence
\begin{equation}\label{mastarde}
\Phi(r,v)=(n+a)\,(1+C_0r)+2\,r\,(1+C_0r)\,\frac{\mathcal{I}(r,v)}{\mathcal{H}(r,v)},
\end{equation}
and it is enough to show that $r\,(1+C_0r)\,\frac{\mathcal{I}(r,v)}{\mathcal{H}(r,v)}$ is monotone.
For that purpose we note that, since
$$\mathcal{D}'(r,v)=\frac{n+a-1}{r}\,\mathcal{D}(r,v)-\frac{2}{r}\int_{\BB}{\bigl((x,y)\cdot\nabla v\bigr)\, \divv(|y|^{a}\nabla v)}+2\int_{\partial \BB}{|y|^{a}v_{\nu}^{2}},$$
it follows by \eqref{i} that
\begin{eqnarray*}
\mathcal{I}'(r,v)&=&\frac{n+a-1}{r}\,\mathcal{I}(r,v)-\frac{n+a-1}{r}\int_{\BB}v\, \divv (|y|^{a}\nabla v)\\
& &-\,\frac{2}{r}\int_{\BB}{\bigl((x,y)\cdot\nabla v\bigr)\, \divv(|y|^{a}\nabla v)}+2\int_{\partial \BB}{|y|^{a}v_{\nu}^{2}}+\int_{\partial \BB}{v\, \divv (|y|^{a}\nabla v)}.
\end{eqnarray*}
Thus, recalling that $\divv (|y|^{a}\nabla v)=-L_a v$, by \eqref{h'} and the Cauchy-Schwarz inequality we obtain
$$
\frac{d}{dr}\logg\left(r\,(1+C_0r)\,\frac{\mathcal{I}(r,v)}{\mathcal{H}(r,v)}\right)\geq\frac{C_0}{1+C_0r}-\mathcal{E}(r,v),
$$
where
\begin{equation}\label{E}
\mathcal{E}(r,v):=\frac{-\frac{1}{r}\left(\int_{\BB}{\left[2\bigl((x,y)\cdot\nabla v\bigr)+(n+a-1)v\right]L_a v}\right)+\int_{\partial \BB}{vL_a v}}{\mathcal{I}(r,v)}.
\end{equation}
Since $\frac{C_0}{1+C_0r}\geq \frac{C_0}2$ provided $r \leq r_0:=\frac{1}{C_0}$
and $C_0$ can be chosen arbitrarily large, to conclude the proof it will be enough to show that $\mathcal{E}(r,v)$ is bounded
independently of~$r$.
For that, we will estimate separately each term of the numerator and denominator of this function.

Since $v$ satisfies \eqref{ext2} outside $\{v=0\}\cap \{y=0\}$ while $v\,L_av=0$ on the set $\{v=0\}\cap \{y=0\}$
(because $L_av$ is a signed measure),
using the Cauchy-Schwarz inequality, \eqref{i}, and Lemma \ref{auxiliar_frecuency}, we obtain that
\begin{eqnarray}
\mathcal{I}(r,v)&=&\mathcal{D}(r,v)-\int_{\BB} v\,L_av =\mathcal{D}(r,v)-\int_{\BB\setminus \{v=0\}} v\,L_av\nonumber\\
&\geq& \mathcal{D}(r,v)-2\left(\int_{\BB}{|y|^{a} v^2}\right)^{1/2}\left(\int_{\BB\setminus \{v=0\}}{|y|^{-a}\left(L_av\right)^2}\right)^{1/2}\label{milveces}\\
&\geq& \mathcal{D}(r,v)-2\,\mathcal{G}(r,v)^{1/2}\left(\int_{\BB}{|y|^{a}|x|^{2(1+\gamma)}}\right)^{1/2}\nonumber\\
&\geq& \mathcal{D}(r,v)-2\,\mathcal{G}(r,v)^{1/2}r^{\frac{n+1+a}{2}+1+\gamma}\nonumber\\
&\geq& \mathcal{D}(r,v)-C\left(\mathcal{D}(r,v)^{1/2}r^{\frac{n+1+a}{2}+2+\gamma}+r^{(n+1)+2(\gamma+2)+a}\right).\label{denominator}
\end{eqnarray}
Similarly, since $(x,y)\cdot\nabla v=0$ on the set $\{v=0\}\cap \{y=0\}$ we get
\begin{equation}\label{numerador1}
\left|\frac{1}{r}\int_{\BB}{\bigl((x,y)\cdot\nabla v\bigr)L_a v}\right|\leq C\, \mathcal{D}(r,v)^{1/2} r^{\frac{n+1+a}{2}+1+\gamma}
\end{equation}
and
\begin{equation}
\max\left\{\left|\frac{1}{r}\int_{\BB} v\, L_av\right|, \left|\int_{\partial \BB}{v\,L_av}\right|\right\}
\leq C\,\Bigl(\mathcal{D}(r,v)^{1/2} r^{\frac{n+1+a}{2}+1+\gamma}+r^{n+2(\gamma+2)+a}\Bigr).\label{numerador2}
\end{equation}
Thus, it follows by \eqref{E}-\eqref{numerador2} that
$$
|\mathcal{E}(r,v)|\leq C\, \frac{\mathcal{D}(r,v)^{1/2}r^{\frac{n+1+a}{2}+1+\gamma}+r^{n+2(\gamma+2)+a}}{\mathcal{D}(r,v)-C\left(\mathcal{D}(r,v)^{1/2}r^{\frac{n+1+a}{2}+2+\gamma}+r^{(n+1)+2(\gamma+2)+a}\right)}.
$$
Now, recalling that $\mathcal{H}(r,v)>r^{n+a+4+2\gamma}$,
thanks to $\eqref{H1}$ we get
\begin{equation}\label{D1}
\mathcal{D}(r,v)\geq c \,r^{n+a+3+2\gamma},
\end{equation}
and  the previous inequality implies that
$$|\mathcal{E}(r,v)|\leq \frac{C\,\mathcal{D}(r,v)^{1/2}r^{\frac{n+1+a}{2}+1+\gamma}}{\mathcal{D}(r,v)}.$$
Thanks to \eqref{D1} we finally obtain that $|\mathcal{E}(r,v)|\leq C$, as desired.
\end{proof}

\begin{rem}\label{pita}
We note here that, from the computations done in the previous proof (see \eqref{i} and \eqref{mastarde}), if $\mathcal{H}(r,v)\geq Cr^{n+a+4+2\gamma}$ then
\begin{equation}\label{igualdad_imp}
\Phi(r,v)= (1+C_0r)\left((n+a)+2\,\mathcal{N}(r,v)-2r\,\frac{\int_{\BB}{v\, L_av}}{\mathcal{H}(r,v)}\right),
\end{equation}
where
$$\mathcal{N}(r,v):=\frac{ r\,\int_{\BB}{|y|^{a}|\nabla v|^{2}}}{\int_{\partial \BB}{|y|^{a}v^2 }}=\frac{r \,\mathcal{D}(r,v)}{\mathcal{H}(r,v)}$$
is the classic Almgren's frequency function.
As we shall see in the next section, thanks to the non-degeneracy condition that we proved in the previous section,
the last term in \eqref{igualdad_imp} goes to zero as $r\downarrow0$ and therefore
\[\Phi(0^+,v)=n+a+2\,\mathcal{N}(0^+,u).\]
\end{rem}

\section{Blow-ups}
\label{sec4}

In this section we will use the Almgren-type monotonicity formula and the non-degeneracy results of the previous sections to show that, at any free boundary point, there exists a blow-up $v_0$ which is homogeneous and whose degree is either $m=1+s$ or $m=2$.

\begin{prop}\label{phi-u}
Let $v,r_0$ be as in Proposition \ref{Almgren}.
Then
\[\Phi(0^{+},v)=n+a+2m\]
with
\[m=1+s\qquad\textrm{or}\qquad m=2.\]
Moreover there exists a constant $\bar C>0$, independent of the free boundary point $x_0$,
such that
\begin{equation}\label{ineq-H1}
\mathcal{H}(r,v)\leq \bar C\,r^{n+a+2m}\qquad \forall\,r \in (0,r_0),
\end{equation}
and for every $\varepsilon>0$ there exists $r_{\varepsilon,x_0}>0$ such that
\begin{equation}\label{ineq-H2}
 \mathcal{H}(r,v)\geq r^{n+a+2m+\varepsilon} \qquad \forall\,r \in (0,r_{\varepsilon,x_0}).
\end{equation}
\end{prop}

\begin{proof}
Let $m$ be such that
$${\Phi(0^+,v)}=n+a+2m.$$
We claim that $m\leq 2$.

Indeed, since $\Phi$ is monotone nondecreasing (by Proposition \ref{Almgren}), it follows by the definition of $m$ and $\Phi$ that, for every $\varepsilon>0$,
$$n+a+2m\leq \bigl(r+C_0r^2\bigr)\,\frac{d}{dr}\log \max\left\{\mathcal{H}(r,v),\ r^{n+a+4+2\gamma}\right\}\leq n+a+2m+\frac{\varepsilon}{2},$$
for $r$ sufficiently small (more precisely, while the first inequality holds for all $r \in (0,r_0)$, for the second inequality
one needs to take $r$ small enough, the smallness possibly depending both on $x_0$ and $\varepsilon$).
Integrating with respect to $r$, this implies that there exists a constant $C_1$ such that
\begin{equation}\label{tormenta1}
\log r^{n+a+2m}+C_1\geq  \log \max\left\{\mathcal{H}(r,v),\ r^{n+a+4+2\gamma}\right\} \geq \log r^{n+a+2m+\varepsilon/2}-C_1.
\end{equation}
In particular the first inequality above yields the validity of
\eqref{ineq-H1} with $\bar C=e^{C_1}$.

Integrating \eqref{ineq-H1} with respect to $r$ and recalling that $\mathcal H=\mathcal G'$ (see \eqref{HG}) we obtain
$$\int_{\BB(x_0,0)}|y|^{a}|v(x,y)|^2\,dx\,dy\leq C \,r^{n+a+2m+1}\qquad \forall\, r\in (0, r_0),$$
that combined with \eqref{tormenta} proves that $m\leq 2$.

Assuming now without loss of generality that $\varepsilon<2\gamma$, since $m \leq 2$ we see that
the inequality $r^{n+a+4+2\gamma} \leq r^{n+a+4+\varepsilon}\leq e^{-C_1} r^{n+a+2m+\varepsilon/2}$ holds for $r \ll 1$, so \eqref{ineq-H2} follows by the second inequality in \eqref{tormenta1}.

To conclude the proof we notice that, if $m<2$, we can take $\varepsilon>0$ such that $2m+\varepsilon<4$.
In this way, if we set
\begin{equation}\label{dr}
d_r:=\left(\frac{\mathcal{H}^{x_0}(r,v)}{r^{n+a}}\right)^{1/2},
\end{equation}
 it follows by \eqref{ineq-H2} that
\begin{equation}\label{condicionCSS}
\lim\inf_{r\to 0}\frac{d_r}{r^{2}}=\infty,
\end{equation}
and \cite[Lemma 6.2]{CSS} shows
that the only possible homogeneity for a blow-up of $v$ is $1+s$.
\end{proof}

\begin{rem}\label{rmk:vLv}
Notice that, as in \eqref{milveces}, the Cauchy-Schwarz inequality and \eqref{ext2} yields
\begin{equation}\label{milveces2}
\left|\int_{\BB(x_0,0)}{v\, L_av}\right|\leq C\,\mathcal{G}(r,v)^{1/2}r^{\frac{n+1+a}{2}+1+\gamma}.
\end{equation}
Also, since $\mathcal G'=\mathcal H$, it follows by \eqref{ineq-H1} that
$$
\mathcal{G}(r,v)\leq \bar C\,r^{n+a+2m+1},
$$
therefore
\begin{equation}\label{vLv}
\left|\int_{\BB(x_0,0)}{v\, L_av}\right|\leq C\,r^{n+a+m+2+\gamma} \qquad \forall\,r \in (0,r_0).
\end{equation}
Since $m\leq 2$, choosing $\varepsilon\leq \gamma$ in \eqref{ineq-H2} we see that
\begin{equation}\label{vLv2}\lim_{r\to 0^+}{2\,r\,\frac{\int_{\BB(x_0,0)}{v\, L_av}}{\mathcal{H}(r,v)}}=0.\end{equation}
Therefore, as we announced in Remark \ref{pita}, taking the limit as $r \to 0^+$ in \eqref{igualdad_imp} we obtain
\begin{equation}\label{igualdad_imp2}
\Phi(0^+,v)=(n+a)+2\,\mathcal{N}(0^+,v),
\end{equation}
that is, the value $m$ in Proposition \ref{phi-u} coincides with the value of the classic Almgren's formula for the point $x_0\in\Gamma(u)$.
\end{rem}

We next show the following:

\begin{prop}\label{phi-u2}
Let $v=v^{x_0}$ be as in Proposition \ref{Almgren}, set
\begin{equation}\label{m_}
m:=\frac{\Phi^{x_0}(0^{+}, v)-n-a}{2},
\end{equation}
and let
\[v_r(x,y)=v^{x_0}_r(x,y):=\frac{v(x_0+rx,ry)}{d_r}\]
be a blow-up sequence, where $d_r$ is defined in \eqref{dr}.
Then, up to a subsequence, $v_r$ converge as $r \to 0^+$ to a homogeneous function $v_0$, which is nonnegative in $\{y=0\}$ and homogeneous of degree $m$.

Moreover:
\begin{itemize}
\item[(a)] either
\[m=1+s\qquad \textrm{and}\qquad v^{}_0(x,0)=c\,((x-x_0)\cdot\nu)_{+}^{1+s}\]
for some $\nu\in \mathbb{S}^{n-1}$ and some positive constant $c$;
\item[(b)] or
\[m=2,\qquad v^{}_0(x,0)\ \textrm{is a polynomial of degree 2},\]
and $x_0$ is a singular point.
\end{itemize}
\end{prop}

\begin{proof}
We follow the proof of \cite[Lemma 6.2]{CSS} (see also \cite{GP}).

First, without loss of generality we can assume that $x_0=0\in\Gamma(u)$. Since $\mathcal{H}(r,v)\geq r^{n+a+4+2\gamma}$ for $r$ small
(this follows by \eqref{ineq-H2} with $\varepsilon=2\gamma$)
and the frequency function $\Phi(r, v)=\Phi^{x_0}(r, v)$ is monotone nondecreasing (by Proposition \ref{Almgren}),
using \eqref{i} we get
$$r\,\frac{\int_{\BB}|y|^{a}|\nabla v|^{2}-\int_{\BB}v\, L_av}{\mathcal{H}(r,v)}\leq \Phi(r,v)\leq \Phi(r_0,v)\leq C,\qquad 0<r<r_{\varepsilon,x_0},$$
and
\eqref{vLv2} yields
$$r\,\frac{\int_{\BB}|y|^{a}|\nabla v|^{2}}{\mathcal{H}(r,v)}\leq C,\qquad 0<r\ll 1.$$
Taking into account the definition of $v_r$ the previous inequality is equivalent to
$$\int_{\BBone}{|y|^{a}|\nabla v_r|^{2}}\leq C,\qquad 0<r\ll 1.$$
Also, it follows by the definition of $d_r$ that $\|v_r\|_{L^{2}(\partial \BBone,|y|^a)}=1$.

This implies that the sequence $\{v_r\}$ is uniformly bounded in the Hilbert space $H^{1}(\BBone,|y|^{a})$,
so, up to a subsequence that we will still denote by $\{v_r\}$,  there exists a function $v_0\in H^{1}(\BBone,|y|^{a})$ such that
\begin{eqnarray}
\displaystyle v_{r}&\rightharpoonup& v_0 \qquad
\mbox{ weakly in } H^{1}(\BBone,|y|^{a}), \nonumber\\
\displaystyle v_{r}&\to& v_0\qquad
\mbox{ strong in } L^{2}(\partial \BBone,|y|^a), \label{estrella}\\
\displaystyle v_r&\to& v_0\qquad \mbox{ a.e. in
} \BBone.\nonumber
\end{eqnarray}
Moreover, it follows from the optimal regularity for the fractional obstacle problem proved in \cite{CSS} that
$\|v_r\|_{C^{1,s}_{\rm loc}(\BBone)}\leq C$.
Hence, applying \cite[Lemma 4.4]{CSS} we conclude that, up to a subsequence,
\begin{equation}\label{conv_fuerte}
\displaystyle v_{r}\to v_0 \qquad
\mbox{ strongly in } H^{1}(\BBone,|y|^{a}) \mbox { and in } C^{1,\alpha}_{\rm loc}(\BBone) \text{ for all $\alpha<s$.}
\end{equation}
We note here that, thanks to \eqref{ext2} and \eqref{ineq-H2}, we have
\begin{eqnarray*}
|L_a v_r(x,y)|& \leq&
C\, \frac{r^2}{d_r}\,r^{1+\gamma}\,|y|^a|x|^{1+\gamma}\\
& \leq& C\,r^{3+\gamma-m-\varepsilon/2}\,|y|^a|x|^{1+\gamma}
\qquad \text{outside $\{v_r=0\}\cap \{y=0\}$}
\end{eqnarray*}
for $r$ sufficiently small.
So, since $m\leq 2$, choosing $\varepsilon \leq \gamma$ and letting $r \to 0^+$ we get
\begin{equation}\label{sol_global}
\left\{ \begin{array}{rcll}
v_0(x,0)&\geq&0  &\textrm{in}\ {\mathbb{R}^{n}},\\
L_a v_0(x,y)&=& 0&\textrm{in}\  {\mathbb{R}^{n+1}}\setminus\{(x,0):\, v_0(x,0)=0\},\\
L_a v_0(x,y)&\geq& 0&\textrm{in}\  {\mathbb{R}^{n+1}}.
\end{array}\right.
\end{equation}
In addition, since $\|v_r\|_{L^{2}(\partial \BBone,|y|^a)}=1$, \eqref{estrella} implies that
\begin{equation}\label{nonzero}
\mbox{ $\|v_0\|_{L^{2}(\partial \BBone,|y|^a)}=1, \quad$ so in particular $v_0\not \equiv 0$.}
\end{equation}
Now, thanks to \eqref{igualdad_imp2}-\eqref{conv_fuerte} we can take the limit in the frequency formula and get
$$\mathcal{N}(\rho,v_0)=\lim_{r\to 0^+}\mathcal{N}(\rho, v_r)=\lim_{r\to 0^+}\mathcal{N}(r\rho, v)=m.$$
This implies that the classical Almgren's frequency formula $\mathcal N(\cdot,v_0)$ is constant,
hence $v_0$ is a homogeneous function of degree $m$ in $\mathcal{B}_{1/2}$ (see \cite[Theorem 6.1]{CSext}).
Also, by Proposition \ref{phi-u} we know that $m=1+s$ or $m=2$.
We now distinguish between these two cases.
\smallskip

If $m=1+s$, since \eqref{condicionCSS} is satisfied, \cite[Lemma 6.2 and Proposition 5.5]{CSS} imply that $v_0(x)=c(x\cdot\nu)_{+}^{1+s}$, for some $\nu\in \mathbb{S}^{n-1}$.
\smallskip

If $m=2$ we now show that  $0=x_0$ is a singular point and that $v_0$ is a homogeneous polynomial of degree 2.
For this we suitably modify an argument used in \cite[Theorem 1.3.2]{GP} for the Signorini problem.

Let us consider $\lambda:=\frac{n}{1+a}$ so that $P(x,y):=|x|^{2}-\lambda |y|^{2}$ satisfies $L_a P(x,y)=0$ in $\mathbb{R}^{n+1}$. By \eqref{sol_global} we see that
the nonnegative measure $\mu:=L_a v_0$ is supported on $\{y=0\}$.
In addition, since $v_0$ and $P$ are homogeneous of degree $2$,
it holds
$$
(x,y)\cdot\nabla P(x,y)=2\,P(x,y),\qquad (x,y)\cdot\nabla v_0(x,y)=2\,v_0(x,y) \qquad \forall\,(x,y),
$$
hence if $\Psi\in C^{\infty}_{0}(\mathbb{R}^{n+1})$ is a radial nonnegative function
we get
\begin{eqnarray*}
v_0(x,y)\,\nabla \Psi(x,y)\cdot\nabla P(x,y)&=&2\,v_0(x,y)\,P(x,y)\,\nabla\Psi(x,y)\cdot \frac{(x,y)}{|x|^2+|y|^2}\\
& =&
P(x,y)\,\nabla\Psi(x,y)\cdot\nabla v_0(x,y)\qquad\qquad \qquad \forall\,(x,y).
\end{eqnarray*}
Also, since $\mu$ is supported on $\{y=0\}$ and $P=|x|^2\geq 0$ on $\{y=0\}$ we see that
$$0\leq \langle \mu, \Psi\, |x|^2\rangle= \langle \mu, \Psi\, P\rangle.$$
Thus, integrating by parts and using that $L_aP=0$ we obtain
\begin{eqnarray}
0\leq \langle \mu, \Psi\, |x|^2\rangle&=&\langle L_av_0, \Psi\, P\rangle=\int_{\mathbb{R}^{n+1}}|y|^{a}\nabla v_0\cdot\nabla(\Psi\, P)\nonumber\\
&=&\int_{\mathbb{R}^{n+1}}|y|^{a}\left(\Psi\,\nabla P\cdot\nabla v_0+P\,\nabla v_0\cdot\nabla\Psi\right)\nonumber\\
&=&\int_{\mathbb{R}^{n+1}}{|y|^{a}\left(\Psi\,v_0\,L_aP-v_0\,\nabla\Psi\cdot\nabla P+P\,\nabla\Psi\cdot\nabla v_0\right)}=0\nonumber,
\end{eqnarray}
that is
$$0=\int_{\{y=0\}}|x|^{2}\,\Psi(x,0)\,d\mu(x).$$
Since $\Psi \geq 0$ is arbitrary it follows by the equation above that $\mu=c\,\delta_0$ for some $c\geq 0$.
However, since $\mu$ is $0$-homogeneous
(being a second order derivative of a 2-homogeneous function), the only possibility is that $\mu \equiv 0$, that is $L_a v_0=0$ in the whole $\mathbb{R}^{n+1}$.
Applying \cite[Lemma 2.7]{CSS}, we can then conclude that $v_0$ is a polynomial of degree $2$.

Being a polynomial, the set $\{v_0=0\}\cap \{y=0\}$ cannot have positive measure unless $v_0(\cdot,0)\equiv 0$, which would imply that $v_0\equiv 0$ in $\R^{n+1}$,
a contradiction to \eqref{nonzero}. This proves that the contact set $\{v_0=0\}\cap \{y=0\}$ has measure zero
for any possible blow-up $v_0$,
which combined with \eqref{conv_fuerte} implies that \eqref{singular} holds for $x_0=0$.
\end{proof}

\begin{rem}\label{rmk:split Gamma}
Let us note here that, thanks to Proposition \ref{phi-u}, we can define
\begin{equation}\label{Gamma_1+s}
\Gamma_{1+s}(u):=\left\{x_0\in\Gamma(u):\, \Phi^{x_0}(0^{+},v)=n+a+2(1+s)\right\},
\end{equation}
\begin{equation}\label{Gamma_2}
\Gamma_{2}(u):=\left\{x_0\in\Gamma(u):\, \Phi^{x_0}(0^{+},v)=n+a+4\right\},
\end{equation}
and the decomposition $\Gamma(u)=\Gamma_{1+s}(u)\cup \Gamma_{2}(u)$ holds.
Also, again by Proposition \ref{phi-u2} we see that the set $\Gamma_{1+s}(u)$ consists of regular points, and $\Gamma_2(u)$ consists of singular points.
Furthermore, because the map $x_0\mapsto \Phi^{x_0}(0^{+},v^{x_0})$ can be written as the infimum over $r$ of the continuous maps
$x_0\mapsto \Phi^{x_0}(r,v^{x_0})$, we deduce that $x_0\mapsto \Phi^{x_0}(0^{+},v^{x_0})$ is upper-semicontinuous,
hence $\Gamma_{1+s}(u)$ (resp. $\Gamma_2(u)$) is an open (resp. closed) subset of $\Gamma(u)$.
Finally, since
the contact set $\{u=\varphi\}$ is compact (by our assumption on the positivity set of $\varphi$, see \eqref{obstacle g}-\eqref{obstacle l}),
both $\Gamma(u)$ and $\Gamma_2(u)$ are compact sets.
\end{rem}

\section{Monneau-type monotonicity formula}
\label{sec5}

In the previous section we showed that free boundary points belong either to $\Gamma_{1+s}(u)$ or to $\Gamma_2(u)$.
Our goal here is to establish a Monneau-type monotonicity formula that will be later used to establish uniqueness of blow-ups for points in $\Gamma_2(u)$.
This Monneau-type monotonicity formula, stated next, extends the one established in \cite{GP} for $s=\frac12$.
Our proof essentially follows the arguments in \cite{GP}, although we slightly simplify some of the computations.

From now on we denote by $\mathfrak P_2$ the set of $2$-homogeneous quadratic polynomials $p_2(x,y)$ satisfying
\[L_ap_2=0\quad \textrm{in}\ \R^{n+1},\qquad p_2\geq0\quad \textrm{for}\ \{y=0\},\qquad p_2(x,y)=p_2(x,-y),\]
that is
\[\mathfrak{P}_2:=\left\{p_2(x,y)=\langle Ax, x\rangle-by^2\,:\, \mbox{$A\in \R^{n\times n}$ symmetric, $A \geq 0$, $A\not\equiv0$,}\,  L_ap_2=0\right\}.\]
Our Monneau-type monotonicity formula reads as follows.

\begin{prop}\label{monneau}
Let $u$ solve:
\begin{itemize}
\item[(A)] either the obstacle problem \eqref{extpb}, with $\varphi$ satisfying~\eqref{obstacle g};
\item[(B)] or the obstacle problem \eqref{pb-sign}, with $\varphi$ satisfying~\eqref{obstacle l}.
\end{itemize}
Then there exists a constant $C_M>0$ such that the following holds:

Let $x_0\in \Gamma_2(u)$, let $v=v^{x_0}$ be defined as in \eqref{v}, and let $p_2\in \mathfrak{P}_2$.
Also, let $\gamma>0$  be as in \eqref{obstacle g}-\eqref{obstacle l}.
Then  the quantity
\[\mathcal{M}^{x_0}(r,v, p_2):=\frac{1}{r^{n+a+4}}\int_{\partial \BB(x_0,0)}|y|^{a}\bigl(v(x,y)-p_2(x-x_0,y)\bigr)^2\]
satisfies
\[\frac{d}{dr}\mathcal{M}^{x_0}(r,v, p_2)\geq -C_M\,r^{\gamma-1}\qquad \forall\,r \in (0,r_0),\]
where $r_0$ is as in Proposition \ref{Almgren}.
\end{prop}

The rest of this Section is devoted to the proof of Proposition \ref{monneau}.
For that we will need the following lower bound on a suitable Weiss-type energy:

\begin{lem}\label{bound-weiss}
Let $v$ be as in Proposition \ref{monneau}, and let $r_0$ be as in Proposition \ref{Almgren}.
Then there exists a constant $C_W>0$ such that the following holds:

The quantity
\[\mathcal{W}^{x_0}(r,v):=\frac{1}{r^{n+a+3}}\int_{\BB(x_0,0)}|y|^{a}|\nabla v|^2-\frac{2}{r^{n+a+4}}\int_{\partial \BB(x_0,0)}|y|^{a}v^2\]
satisfies
\begin{equation}\label{ineq-weiss}
\mathcal{W}^{x_0}(r,v)\geq -C_W\,r^{\gamma}\qquad \forall\, r \in (0,r_0).
\end{equation}
\end{lem}

\begin{proof}
We will use the Almgren-type monotonicity formula obtained in the previous sections.
Since there is not possible confusion along this proof, we will use the notation $\Phi=\Phi^{x_0}$, $\mathcal{H}=\mathcal{H}^{x_0}$, $\mathcal{I}=\mathcal{I}^{x_0}$, and $\mathcal{D}=\mathcal{D}^{x_0}$.

First, by definition of $\Gamma_2(u)$ (see \eqref{Gamma_2}) we see that $\Phi(0^{+},v)=n+a+4$.
Thus, by the monotonicity of $\Phi(\cdot,v)$ on $(0,r_0)$ (see Proposition \ref{Almgren}), for any $r \in (0,r_0)$ we have that either
\begin{equation}\label{1}
\Phi(r,v)=\bigl(r+C_0r^2\bigr)\,\frac{\mathcal{H}'(r,v)}{\mathcal{H}(r,v)}\geq n+a+4
\end{equation}
or
\begin{equation}\label{2}
\mathcal{H}(r,v)\leq r^{n+a+4+2\gamma}.
\end{equation}
We split the proof of \eqref{ineq-weiss} in two cases.

\vspace{2mm}

\noindent
\emph{- Case 1}. If \eqref{1} holds then it follows by \eqref{h'} that
$$
(r+C_0r^2)\,\left(\frac{n+a}{r}+2\,\frac{{\mathcal{I}(r,v)}}{\mathcal{H}(r,v)}\right)\geq n+a+4,
$$
that is
$$
n+a+2\,r\,\frac{{\mathcal{I}(r,v)}}{\mathcal{H}(r,v)}\geq
 n+a+4- C_0r^2\,\left(\frac{n+a}{r}+2\,\frac{{\mathcal{I}(r,v)}}{\mathcal{H}(r,v)}\right),
$$
and since $r\left(\frac{n+a}{r}+2\,\frac{{\mathcal{I}(r,v)}}{\mathcal{H}(r,v)}\right)\leq \Phi(r,v)\leq C$ we get
$$
r\,\frac{{\mathcal{I}(r,v)}}{\mathcal{H}(r,v)}\geq 2-\frac{C_0}{2}\,r^2\left(\frac{n+a}{r}+2\frac{{\mathcal{I}(r,v)}}{\mathcal{H}(r,v)}\right)\geq 2-C\,r.
$$
Hence, since $m=2$, recalling \eqref{i}, \eqref{vLv}, and \eqref{ineq-H1} we obtain
$$C\,r^{n+a+5+\gamma}+\bigl(r \,\mathcal{D}(r,v)-2\,\mathcal{H}(r,v)\bigr)\geq -C\,r\,\mathcal{H}(r,v) \geq -C\,r^{n+a+5},$$
which gives that
$\mathcal{W}^{x_0}(r,v)\geq -C\,r \geq -C\,r^{\gamma}$ (as $\gamma \leq 1$), as desired.

\vspace{2mm}

\noindent
\emph{- Case 2}. If \eqref{2} holds then we simply use that $\mathcal{D}(r,v)\geq0$ to obtain
\[\frac{1}{r^{n+a+3}}{\mathcal{D}(r,v)}-\frac{2}{r^{n+a+4}}\mathcal{H}(r,v)\geq -C\,r^{2\gamma}\geq -C\,r^{\gamma},\]
which concludes the proof of \eqref{ineq-weiss}.
\end{proof}

We can now prove Proposition \ref{monneau}.

\begin{proof}[Proof of Proposition \ref{monneau}]
Without loss of generality we can assume $x_0=0$. Set $w:=v-p_2$ and let us use the notation $z=(x,y) \in \R^{n+1}$.
Then
\begin{eqnarray}
\frac{d}{dr}\mathcal{M}^{x_0}(r,v, p_2)&=&\frac{d}{dr}\int_{\partial \BBone}\frac{|y|^{a}|w(rz)|^2}{r^4}\nonumber\\
&=&\int_{\partial \BBone}|y|^{a}\,\frac{2\,w(rz)\bigl(rz\cdot\nabla w(rz)-2\,w(rz)\bigr)}{r^5}\nonumber\\
&=&\frac{2}{r^{n+a+5}}\int_{\partial \BB}|y|^{a}w\bigl(z\cdot\nabla w-2\,w\bigr).\label{d dr M}
\end{eqnarray}
We now claim that
\begin{equation}\label{claim-weiss}
\mathcal{W}^{x_0}(r,v)\leq \frac{1}{r^{n+a+4}}\int_{\partial \BB}|y|^{a}w(z\cdot \nabla w-2w)+C\,r^{1+\gamma}.
\end{equation}
Indeed, since $L_ap_2=0$ in $\R^{n+1}$ and $p_2$ is $2$-homogeneous,
it is easy to check that $\mathcal{W}^{x_0}(r,p_2)\equiv0$. Hence, using again that $L_ap_2=0$
and that $z\cdot \nabla p_2=2\,p_2$ (by the $2$-homogeneity), integrating by parts we get
\begin{eqnarray}
\mathcal{W}^{x_0}(r,v)&=&\mathcal{W}^{x_0}(r,v)-\mathcal{W}^{x_0}(r,p_2)\\
&=&\frac{1}{r^{n+a+3}}\int_{\BB}|y|^{a}\left(|\nabla w|^2+2\,\nabla w\cdot\nabla p_2\right)-\frac{2}{r^{n+a+4}}\int_{\partial \BB}|y|^{a}\left(w^2+2\,w\,p_2\right)\nonumber\\
&=&\frac{1}{r^{n+a+3}}\int_{\BB}|y|^{a}|\nabla w|^2+\frac{1}{r^{n+a+3}}\int_{\BB}2\,w\,L_ap_2\nonumber\\
&& +\frac{2}{r^{n+a+4}}\int_{\partial \BB}|y|^{a}w\left(z\cdot \nabla p_2-2\,p_2\right)- \frac{2}{r^{n+a+4}}\int_{\partial \BB}|y|^{a}w^2\nonumber\\
&=&\frac{1}{r^{n+a+3}}\int_{\BB}|y|^{a}|\nabla w|^2-\frac{2}{r^{n+a+4}}\int_{\partial \BB}|y|^{a}w^2\label{split}.
\end{eqnarray}
Using now that $p_2\leq C\,r^2$ in $\BB$ and arguing as we did in Remark \ref{rmk:vLv} to obtain \eqref{vLv}, we get
$$
\left|\int_{\BB}{w\, L_a w}\right|=\left|\int_{\BB}{(p_2-v)\, L_a v}\right|\leq C\, r^{n+a+4+\gamma},$$
where for the first equality we used again that $L_ap_2=0$. Integrating by parts in \eqref{split} and using the previous bound, we conclude that
\begin{eqnarray*}
\mathcal{W}^{x_0}(r,v)&=&\frac{1}{r^{n+a+3}}\int_{\BB}w\,L_aw+\frac{1}{r^{n+a+4}}\int_{\partial \BB}|y|^{a}w\,(z\cdot \nabla w-2\,w)\\
&\leq& \frac{1}{r^{n+a+4}}\int_{\partial \BB}|y|^{a}w(z\cdot \nabla w-2w)+ C\,r^{1+\gamma},
\end{eqnarray*}
and \eqref{claim-weiss} follows.

Finally, combining \eqref{d dr M} and \eqref{claim-weiss} and using Lemma \ref{bound-weiss} we get
$$\frac{d}{dr}\mathcal{M}^{x_0}(r,v,p_2) \geq \frac{2}{r}\,\mathcal{W}^{x_0}(r,v)-C\,r^{1+\gamma}\geq -C\,r^{\gamma-1}$$
and the proposition is proved.
\end{proof}

\section{Uniqueness of blow-ups and proof of Theorem \ref{thm1}}
\label{sec6}

We saw in the previous sections that blow-ups are homogeneous of order $m$, and either $m=1+s$ or $m=2$.
When the blow-up at $x_0$ is of order $m=1+s$, it follows by \cite[Theorem 7.7]{CSS} that $x_0$ is a regular point and that the free boundary is
a $C^{1+\alpha}$ $(n-1)$-dimensional surfaces in a neighborhood of $x_0$.
When the blow-up at $x_0$ is of order $m=2$, then $x_0$ belongs to the set of singular points, but we still have not proved anything about the regularity of this set.

By Proposition \ref{phi-u2} we know that all blow-ups of the function $v^{x_0}$ at $x_0\in\Gamma_2(u)$ are homogeneous polynomials of order 2.
However, it may happen that one gets different polynomials over different subsequences.
We prove in this section that this does not happen, i.e., we show uniqueness of the blow-up.
Moreover, we also prove continuity of the blow-up profiles with respect to $x_0$.
This will yield the regularity of the set $\Gamma_2(u)$, and thus to our main result Theorem \ref{thm1}.

We start with the following.

\begin{lem}\label{est-2}
Let $u$ solve:
\begin{itemize}
\item[(A)] either the obstacle problem \eqref{extpb}, with $\varphi$ satisfying~\eqref{obstacle g};
\item[(B)] or the obstacle problem \eqref{pb-sign}, with $\varphi$ satisfying~\eqref{obstacle l}.
\end{itemize}
Let $r_1$ be as in Lemma \ref{nondegeneracy0} or Lemma \ref{nondegeneracy},
let $r_0$ be as in Proposition \ref{Almgren}, and set $\hat r:=\min\{r_0/2,r_1\}$.
Then there exist constants $C_+,c_->0$ such that the following holds:

Let $x_0\in \Gamma_2(u)$, and
let $v=v^{x_0}$ be defined as in \eqref{v}. Then
\[c_-\,r^2\leq \sup_{\BB(x_0,0)}|v^{}|\leq C_+r^2\qquad \forall\,r \in (0,\hat r).\]
\end{lem}

\begin{proof}
The lower bound was already established  for every $r\in (0, r_1)$ in Corollary~\ref{non-degeneracy-v}.

To prove the upper bound we may assume that $x_0=0$. Then, since $m=2$, \eqref{ineq-H1} yields
\[\mathcal{H}(r,v)=\int_{\partial\BB}|y|^a\, v^2\leq C\,r^{n+a+4}\qquad \forall\,r \in (0,r_0),\]
and it follows by \eqref{p18} and Lemma \ref{fabes-kenig} that
\[\sup_{\BB}v^+\leq C\,r^2\qquad \forall\,r \in (0,r_0/2).\]
Repeating the same argument with $v^-$ in place of $v^+$ we also get
\[\sup_{\BB}v^-\leq C\,r^2\qquad \forall\,r \in (0,r_0/2),\]
and the lemma is proved.
\end{proof}

Define
\[\mathbb{P}_2^+:=\left\{p_2(x)=\langle Ax,x\rangle\,:\, \mbox{$A\in \R^{n\times n}$ symmetric, $A \geq 0$, $A\not\equiv0$}\right\}.\]
We now prove the uniqueness and continuity of the blow-ups (compare with \cite[Theorems 2.8.3 and 2.8.4]{GP}).

\begin{prop}\label{prop-singular-set}
Let $u$ solve:
\begin{itemize}
\item[(A)] either the obstacle problem \eqref{extpb}, with $\varphi$ satisfying~\eqref{obstacle g};
\item[(B)] or the obstacle problem \eqref{pb-sign}, with $\varphi$ satisfying~\eqref{obstacle l}.
\end{itemize}
Then there exists a modulus of continuity $\omega:\R^+\to \R^+$ such that, for any $x_0\in \Gamma_2(u)$, we have
\[u(x)-\varphi(x)=p_2^{x_0}(x-x_0)+\omega\bigl(|x-x_0|\bigr)|x-x_0|^2\]
for some polynomial $p_2^{x_0}\in \mathbb{P}_2^+$.
In addition the mapping $\Gamma_2(u)\ni x_0\mapsto p_2^{x_0} \in \mathbb{P}_2^+$ is continuous, with
$$
\int_{\partial\mathcal{B}_1}|y|^a\bigl(p_2^{x_0'}-p_2^{x_0}\bigr)^2 \leq \omega(|x_0-x_0'|)\qquad \forall\,x_0,x_0' \in \Gamma_2(u).
$$
\end{prop}

\begin{proof}
Let $v^{x_0}$ be given by \eqref{v}, and define
\[{v_r^{x_0}(x,y)=\frac{v^{x_0}\bigl(x_0+rx,\,ry\bigr)}{r^2}}.\]
By Lemma \ref{est-2} we have that
\[c_-\rho^2\leq \sup_{\mathcal B_\rho} |v_r^{x_0}|\leq C_+\rho^2,\]
for all $\rho\in (0,\hat r/r)$. Thus, arguing as in Proposition \ref{phi-u2} we get
\begin{equation}\label{aggain}
\mbox{$v_{r_j}^{x_0}\longrightarrow v_0^{x_0}$ in $C^{1,\alpha}_{\rm loc}(\R^{n+1})$ along a subsequence $r_j\downarrow0$,}
\end{equation}
and $v_0^{x_0}$ is not identically zero.
Now, since $\mathcal{N}(0^+,v^{x_0})=2$, it follows that
\[\mathcal{N}(r,v_0^{x_0})=\lim_{r_j\downarrow0}\mathcal{N}(r,v_{r_j}^{x_0})=\lim_{r_j\downarrow0}\mathcal{N}(rr_j,v^{x_0})=\mathcal{N}(0^+,v^{x_0})=2,\]
and so  \cite[Theorem 6.1]{CSext} implies that $v_0^{x_0}$ is homogeneous of degree $2$.
Exactly as in the proof of Proposition \ref{phi-u2} we deduce that $v_0^{x_0}=p_2^{x_0}$ for some polynomial $p_2^{x_0}\in\mathfrak{P}_2$.
Hence, using again \eqref{aggain} we get
\[\mathcal{M}(0^+,v^{x_0},p_2^{x_0})=\lim_{r_j\downarrow0}\int_{\partial\mathcal{B}_1}|y|^a\bigl(v_{r_j}^{x_0}-p_2^{x_0}\bigr)^2=0,\]
and Proposition \ref{monneau} implies that
\begin{equation}\label{Cc}
\int_{\partial\mathcal{B}_1}|y|^a\bigl(v_r^{x_0}-p_2^{x_0}\bigr)^2=\mathcal{M}(r,v^{x_0},p_2^{x_0})\longrightarrow0\qquad\textrm{as}\ r\downarrow0
\end{equation}
(not just along a subsequence!).
This immediately implies that the blow-up is unique, and since $v^{x_0}(x,0)=u(x)-\varphi(x)$,
we deduce that $u(x)-\varphi(x)=p_2^{x_0}(x-x_0)+o\bigl(|x-x_0|^2\bigr)$.
The fact that the rest $o\bigl(|x-x_0|^2\bigr)$ is uniform with respect to $x_0$ follows by a simple compactness argument,
see for instance \cite[Lemma 7.3 and Proposition 7.7]{Sbook}.

We now prove the continuous dependence of $p_2^{x_0}$ with respect to $x_0$.
Given $\varepsilon>0$ it follows by \eqref{Cc} that there exists $r_\varepsilon=r_\varepsilon(x_0)>0$ such that
\[\mathcal{M}(r_\varepsilon,v^{x_0},p_2^{x_0})<\varepsilon.\]
Now, by the continuous dependence of $v^{x_0}$ with respect to $x_0$, there exists $\delta_\varepsilon=\delta_\varepsilon(x_0)>0$ such that
\[\mathcal{M}(r_\varepsilon,v^{x_0'},p_2^{x_0})<2\varepsilon,\]
for all $x_0'\in \Gamma_2(u)$ satisfying $|x_0-x_0'|<\delta_\varepsilon$. Then, assuming without loss of generality that $r_\varepsilon \leq r_0$,
it follows by Proposition \ref{monneau} that
\[\mathcal{M}(r,v^{x_0'},p_2^{x_0})<2\varepsilon+C_M\,r_\varepsilon^{\gamma}\qquad\textrm{for all}\ r\in (0,r_\varepsilon].\]
Hence, letting $r\rightarrow0$ we obtain
$$
\int_{\partial\mathcal{B}_1}|y|^a\bigl(p_2^{x_0'}-p_2^{x_0}\bigr)^2=\mathcal{M}(0^+,v^{x_0'},p_2^{x_0})\leq 2\varepsilon+C_M\,r_\varepsilon^{\gamma},
$$
and by the arbitrariness of $\varepsilon$ and $r_\varepsilon$ we deduce the desired continuity of $p_2^{x_0}$ with respect to~$x_0$.
Finally, the uniform continuity of this map follows from the fact that $\Gamma_2(u)$ is a compact set (being a closed subset of $\Gamma(u)$,
see Remark \ref{rmk:split Gamma}).
\end{proof}

We can finally prove our main theorem, which is just a direct consequence of all our previous results.

\begin{proof}[Proof of Theorem \ref{thm1}]
As shown in Proposition \ref{phi-u2} the free boundary can be decomposed as
$\Gamma(u)=\Gamma_{1+s}(u)\cup \Gamma_2(u)$, where $\Gamma_{1+s}(u)$ consists of regular points where the blow-ups have homogeneity $1+s$,
while $\Gamma_2(u)$ consists of singular points where the blow-ups are homogeneous polynomials of degree $2$.
In addition, as observed in Remark \ref{rmk:split Gamma}, $\Gamma_{1+s}(u)$ is a open subset of $\Gamma(u)$.

As already mentioned before, it has been proved in  \cite[Theorem 7.7]{CSS} that $\Gamma_{1+s}(u)$ is a $(n-1)$-dimensional manifolds of class $C_{\rm loc}^{1,\alpha}$
(and in particular at such points the blow-up is unique, see \cite[Sections 6 and 7]{CSS}).

Concerning the regularity of $\Gamma_2(u)$, it follows by Proposition \ref{prop-singular-set}
that for any $x_0 \in \Gamma_2(u)$ the blow-up of $u-\varphi$ is a unique homogeneous polynomials of degree~$2$, denoted by $p_2^{x_0}(x)=\frac12 \langle A^{x_0}x,x\rangle$, which depends continuously with respect to the blow-up point.
We now stratify $\Gamma_2(u)$ according to the dimension of the kernel of $A^{x_0}$:
$$
\Gamma_2^k(u):=\{x_0 \in \Gamma_2(u)\,:\,{\rm dim}({\rm ker}\,A^{x_0})=k\},\qquad k=0,\ldots,n-1.
$$
Then the same argument as in the case of the obstacle problem for the classical Laplacian (see for instance \cite[Theorem 7.9]{Sbook} and \cite[Theorem 8]{C-obst2},
or \cite[Theorem 1.3.8]{GP})
shows that for any $x_0 \in \Gamma_2^k(u)$ there exists $r=r_{x_0}>0$ such that $\Gamma_2^k(u)\cap B_{r}(x_0)$
is contained in a connected $k$-dimensional $C^1$ manifold, which concludes the proof of Theorem \ref{thm1}.
\end{proof}

\section*{Appendix: The optimal stopping problem}

We provide in this Appendix a brief informal description of the optimal stopping problem and its relation to the obstacle problem \eqref{pb}.

\subsection{Stochastic processes}

Let $X_t$ be a stochastic process in $\R^n$ with
no memory, stationary increments, and satisfying $X_0=0$ a.s.
We notice that if we further assume that $t\mapsto X_t$ is a continuous path a.s. then
$X_t$ must be a \emph{Brownian motion} (possibly with a drift).
However, if we slightly relax this assumption and we only assume stochastic continuity,
then $X_t$ will be a \emph{L\'evy process}.

\subsection{Infinitesimal generators}

The infinitesimal generator of a L\'evy process $X_t$ is an operator $L:C^2(\R^n)\longrightarrow C(\R^n)$ defined by
\[Lu(x):=\lim_{t\downarrow0}\frac{\mathbb{E}\bigl[u(x+X_t)\bigr]-u(x)}{t}.\]
It is a classical fact that this definition leads to the formula
\[\mathbb{E}\bigl[u(x+X_t)\bigr]=u(x)+\mathbb{E}\left[\int_0^tLu(x+X_s)\,ds\right].\]
For a general L\'evy process one has
\[Lu(x)=\textrm{tr}(A\,D^2u)+b\cdot\nabla u+\int_{\R^n}\bigl\{u(x+y)-u(x)-y\cdot\nabla u(x)\chi_{B_1}(y)\bigr\}\,d\nu(y),\]
where $A$ is a non-negative definite matrix, $b\in \R^n$, and $\nu$ is the so-called L\'evy measure satisfying the L\'evy-Khintchine condition $\int_{\R^n}\min\bigl(1,|y|^2\bigr)\,d\nu(y)<\infty$.
We recall that when $X_t$ is the usual Brownian motion then $L=\Delta$,
while if $X_t$ is a stable radially symmetric L\'evy process then $L=-(-\Delta)^s$ for some $s\in(0,1]$.

\subsection{Optimal stopping}

We consider the following classical problem in control theory: given a stochastic process $X_t$, one can decide at each instant
of time whether to stop it or not.
The goal is to discover an optimal choice of the stopping time so that we minimize a cost or maximize a payoff.

More precisely, let $\Omega\subset\R^n$ be a bounded domain and consider the process $x+X_t$ with $x\in \Omega$.
Let $\tau=\tau_x$ be the first exit time from $\Omega$, i.e, the first time at which $x+X_\tau\in\R^n\setminus \Omega$.
Assume we are given a payoff function $\varphi:\R^n\longrightarrow \R$ and that, for any stopping time~$\theta$, we get a payoff
\[J_x[\theta]:=\mathbb{E}\bigl[\varphi(x+X_{\min\{\theta,\tau\}})\bigr].\]
In other words, if we stop at a time $\theta<\tau$ then $x+X_\theta\in \Omega$ and we get a payoff $\varphi(x+X_\theta)$, while
if we do not stop before exiting $\Omega$ (i.e., if $\theta\geq\tau$) then $x+X_\tau\in \R^n\setminus\Omega$ and we get a payoff $\varphi(x+X_\tau)$.
The problem is to find an optimal stopping strategy so that $J_x[\theta]$ is maximized.

For this, we define the value function
\[u(x):=\sup_{\theta} J_x[\theta],\]
and we look for an equation for $u(x)$.
Of course, the optimal $\theta$ will depend on $x\in\Omega$,
and once the function $u(x)$ is known then the optimal stopping time $\theta$ can be found by a dynamic programming argument.

\subsection{Optimality conditions}
Our goal here is to give a heuristic argument to show that the value function $u$
satisfies an obstacle problem with obstacle $\varphi$.
Since this is just a formal argument, we assume that $u$ is smooth.

First,
since in our strategy we can always decide to stop at $x$ and get the payoff $\varphi(x)$,
it follows by the definition of $u$ that
\[u\geq\varphi\qquad\textrm{in}\ \Omega.\]
Consider now a point $x\in\Omega$ such that $u(x)>\varphi(x)$ (that is,
we do \emph{not} stop the process at $x$).
This means that  (for most events $\omega$) we do not stop the process for some time $\delta>0$, and therefore at time $t=\delta$ the process will be at $x+X_\delta$.
But then, by definition of $u$, the best we can do is to get a payoff $u(x+X_\delta)$, and thanks to this fact one can actually show that
\[u(x)= \mathbb E\bigl[u(x+X_\delta)\bigr]+o(\delta)\qquad\textrm{if}\ u(x)>\varphi(x).\]
Recalling now that
\[\mathbb{E}\bigl[u(x+X_\delta)\bigr]=u(x)+\mathbb{E}\left[\int_0^\delta Lu(x+X_s)ds\right]\]
we get
\[\mathbb{E}\left[\int_0^\delta Lu(x+X_s)ds\right]=o(\delta)\qquad\textrm{if}\ u(x)>\varphi(x),\]
so that dividing by $\delta$ and letting $\delta\downarrow0$ we find
\[Lu(x)=0\qquad\textrm{if}\ u(x)>\varphi(x).\]

With similar arguments one can check that $-Lu\geq0$ in $\Omega$, hence $u$ solves the \emph{obstacle problem}
\[\min\bigl\{u-\varphi,\,-Lu\bigr\}=0\qquad\textrm{in}\ \Omega\]
with Dirichlet boundary conditions $u=\varphi$ in $\R^n\setminus\Omega.$

\subsection{Application to finance}
Among several areas where optimal stopping problems arise,
an important one is mathematical finance for American options pricing.

An American option gives an agent the possibility to buy a given asset at a fixed price at any time before the expiration date
(that could also be infinite, in which case the option is called perpetual).
The payoff of this option is a random variable that will depend on the value of this asset at the moment the option is exercised:
If $X_t$ is a stochastic process which represents the price of the assets,
the optimal choice of the moment to exercise the option corresponds to an optimal stopping problem for this process.
By the result of this paper it follows that, for perpetual options, when $X_t$ is a
stable radially symmetric L\'evy process then the exercise region enjoys some very nice geometric structure.

We refer to \cite[Chapter 6,D]{Evans} for a description of the model in the case of Brownian motion, and to the book \cite{CT} for an exhaustive discussion in the case of jump processes (see also \cite{LS,CF} for some regularity results in the case of finite expiration date).

%
%


\begin{thebibliography}{00}


\bibitem{AC}  I. Athanasopoulos, L. Caffarelli, \emph{Optimal regularity of lower dimensional obstacle problems}, Zap. Nauchn. Sem. S.-Peterburg. Otdel. Mat. Inst. Steklov. (POMI) 310 (2004), 49-66.

\bibitem{ACS}  I. Athanasopoulos, L. Caffarelli, S. Salsa, \emph{The structure of the free boundary for lower dimensional obstacle problems}, Amer. J. Math. 130 (2008) 485-498.

\bibitem{C-obst1} L. Caffarelli, \emph{The regularity of free boundaries in higher dimensions.} Acta Math. 139 (1977), no. 3-4, 155-184.

\bibitem{Caf79} L. Caffarelli, \emph{Further regularity for the Signorini problem}, Comm. Partial Differential Equations 4 (1979), 1067-1075.

\bibitem{C-obst2} L. Caffarelli, \emph{The obstacle problem revisited}, J. Fourier Anal. Appl. 4 (1998), 383-402.

\bibitem{CF} L. Caffarelli, A. Figalli, \emph{Regularity of solutions to the parabolic fractional obstacle problem}, J. Reine Angew. Math. 680 (2013), 191-233.

\bibitem{CSS} L. Caffarelli, S. Salsa, L. Silvestre, \emph{Regularity estimates for the solution and the free boundary of the obstacle problem for the fractional Laplacian}, Invent. Math. 171 (2008), 425-461.

\bibitem{CSext} L. Caffarelli, L. Silvestre, \emph{An extension problem related to the fractional Laplacian}, Comm. Partial Differential Equations 32 (2007), 1245-1260.

\bibitem{CDM} J. A. Carrillo, M. G. Delgadino, A. Mellet, \emph{Regularity of local minimizers of the interaction energy via obstacle problems}, Comm. Math. Phys. 343 (2016), 747-781.

\bibitem{CT} R. Cont, P. Tankov, \emph{Financial Modelling With Jump Processes}, Financial Mathematics Series. Chapman \& Hall/CRC, Boca Raton, FL, 2004.

\bibitem{dSS}
D. De Silva, O. Savin,
\emph{Boundary Harnack estimates in slit domains and applications to thin free boundary problems.}  
Rev. Mat. Iberoam. 32 (2016), no. 3, 891-912. 

\bibitem{DSV} S. Dipierro, O. Savin, E. Valdinoci, \emph{All functions are locally $s$-harmonic up to a small error}, J. Eur. Math. Soc. (JEMS) 19 (2017), 957-966.

\bibitem{Evans} L. C. Evans, \emph{An Introduction to Stochastic Partial Differential Equations}, American Mathematical Society, Providence, RI, 2013.

\bibitem{FKS} E. Fabes, C. Kenig, R. Serapioni, \emph{The local regularity of solutions of degenerate elliptic equations}, Comm. Partial Differential Equations 7 (1982), 77-116.

\bibitem{FS} M. Focardi, E. Spadaro, \emph{An epiperimetric inequality for the thin obstacle problem}, Adv. Differential Equations 21 (2016), 153-200.

\bibitem{GP} N. Garofalo, A. Petrosyan, \emph{Some new monotonicity formulas and the singular set in the lower dimensional obstacle problem}, Invent. Math. 177 (2009), 415-461.

\bibitem{GPSVG} N. Garofalo, A. Petrosyan, M. Smit Vega Garcia, \emph{An epiperimetric inequality approach to the regularity of the free boundary in the Signorini problem with variable coefficients}, J. Math. Pures Appl. 105 (2016), 745-787.

\bibitem{GSVG} N. Garofalo, M. Smit Vega Garcia, \emph{New monotonicity formulas and the optimal regularity in the Signorini problem with variable coefficients}, Adv. Math. 262 (2014), 682-750,

\bibitem{JN}
Y. Jhaveri, R. Neumayer,
\emph{Higher regularity of the free boundary in the obstacle problem for the fractional Laplacian}, Adv. Math., to appear.


\bibitem{KRS}
H. Koch, A. R\"uland, W. Shi,
\emph{Higher regularity for the fractional thin obstacle problem}, Preprint 2016.


\bibitem{LS} P. Laurence, S. Salsa, \emph{Regularity of the free boundary of an American option on several assets}, Comm. Pure Appl. Math. 62 (2009), no. 7, 969-994.

\bibitem{MO} S. A. Molchanov, E. Ostrovskii, \emph{Symmetric stable processes as traces of degenerate diffusion processes}, Theory Probab. Appl. 14 (1969), 128-131.

\bibitem{PP} A. Petrosyan, C. A. Pop,
\emph{Optimal regularity of solutions to the obstacle problem for the fractional Laplacian with drift}, J. Funct. Anal. 268 (2015), no. 2, 417-472.

\bibitem{Sbook} A. Petrosyan, H. Shahgholian, N. Uraltseva. \emph{Regularity of free boundaries in obstacle-type problems}, volume 136 of Graduate Studies in Mathematics. American Mathematical Society, Providence, RI, 2012.

\bibitem{S} L. Silvestre, \emph{Regularity of the obstacle problem for a fractional power of the Laplace operator}, Comm. Pure Appl. Math. 60 (2007), 67-112.


\end{thebibliography}
\end{document}